\documentclass[final,leqno]{siamltexmm}

\usepackage{times}
\usepackage{amssymb}
\usepackage{amsfonts}
\usepackage{amsmath}
\definecolor{Green}{rgb}{0,.5,0}
\definecolor{Red}{rgb}{.8,.2,0}
\definecolor{Yellow}{rgb}{.6,.6,.1}
\definecolor{Cyan}{rgb}{.2,.6,.7}
\definecolor{Purple}{rgb}{.4,0,1}
\definecolor{deepred}{rgb}{.53,.29,.24}
\definecolor{Black}{rgb}{0,0,0}
\definecolor{Grey}{rgb}{.45,.45,.45}
\definecolor{Blue}{rgb}{.0,.322,.533}

\newtheorem{remark}{Remark}
\newtheorem{example}{Example}

\title{Storing Cycles in Hopfield-type Networks with Pseudoinverse Learning Rule - Retrievability and Bifurcation Analysis}

\author{Chuan Zhang          \footnotemark[2]
\and    Gerhard Dangelmayr   \footnotemark[2] 
\and    Iuliana Oprea        \footnotemark[2]}

\begin{document}
\maketitle
\newcommand{\slugmaster}{%
\slugger{siads}{xxxx}{xx}{x}{x--x}}

\renewcommand{\thefootnote}{\fnsymbol{footnote}}

\footnotetext[2]{Department of Mathematics, Colorado State University, Fort Collins, Colorado. (\email{zhang@math.colostate.edu}, \email{gerhard@math.colostate.edu}, \email{Juliana@math.colostate.edu}). The work of the first author was partially supported by the Yates Chair Graduate Research Fellowship (2013) awarded by the Department of Mathematics at Colorado State University.}

\begin{abstract}
		In this paper, we study retrievability of admissible cycles and the dynamics of the networks constructed from admissible cycles with the pseudoinverse learning rule. Retrievability of admissible cycles in networks with $C_0>0$ and $\lambda$ sufficiently large are discussed. Based on the linear stability analysis we derive a complete description of all possible local bifurcations of the trivial solution for the networks constructed from admissible cycles. We illustrate numerically that, depending on the structural features, the admissible cycles are respectively stored and retrieved as attracting limit cycles, unstable periodic solutions and delay-induced long-lasting transient oscillations, and the transition from fixed points to the attracting limit cycle bifurcating from the trivial solution takes place through multiple saddle-nodes on limit cycle bifurcations.

\end{abstract}

\begin{keywords}
Cyclic Patterns, Retrievability, Delay, Bifurcations
\end{keywords}

\begin{AMS}\end{AMS}

\pagestyle{myheadings}
\thispagestyle{plain}
\markboth{Chuan Zhang, Gerhard Dangelmayr, and Iuliana Oprea}{Storing Cycles in Neural Networks: Retrievability and Bifurcation Analysis}

\section{Introduction}
Cyclic patterns of neuronal activity are ubiquitous in animal nervous systems, and partially responsible for generating and controlling rhythmic movements such as locomotion, respiration, swallowing and so on. Neural networks that can produce cyclic patterned outputs without rhythmic sensory or central input are called central pattern generators (CPGs). Although in some lower level invertebrate animals the detailed connectivities among the identified CPG neurons have been both experimentally determined and reconstructed \cite{Syed90,Syed91}, both the anatomic structure and dynamics of the CPG networks in most higher vertebrate animals including human beings still remain largely unknown \cite{MackayLyons02, Marder05, Selverston10}. While the experimental identification of the network connectivities is fundamental for fully understanding the CPG networks \cite{Yuste08}, recent experimental observations \cite{Dickinson92,Meyrand94} suggested that CPGs may be highly flexible, and some of them may even be temporarily formed only before the production of motor activity \cite{Jean01}, making the experimental identification of the architecture of CPGs very difficult. Observable movement features such as symmetry etc. have been used as indirect approaches to infer aspects of CPG structures \cite{Golubitsky99}. In \cite{Zhang13}, we started from admissible cycles (the cycles that can be stored in a network with the pseudoinverse learning rule), and studied structural features of admissible cycles and how structural features determine topology of the networks constructed from admissible cycles. 
In this paper, following \cite{Zhang13} and \cite{Zhang14}, we continue to study how the structural features of admissible cycles determine the dynamics of the Hopfield-type networks constructed from the admissible cycles with the pseudoinverse learning rule \cite{Personnaz86}.

Hopfield-type networks with delayed couplings are asymmetric generalizations of the Hopfield networks \cite{Hopfield84}, and have been used as models for generating cyclic patterns \cite{Sompolinsky86,Kleinfeld86,Kleinfeld88}. The dynamics of the Hopfield-type networks with discrete/distributed transmission delay(s) has recently attracted considerable research interest from the dynamical systems theory community (see for example \cite{Campbell06, Campbell99, Campbell05, Chen01, Shih06, Shih07, Guo07, Horikawa09a, Pakdaman97} etc.). Due to its broad applications, ring networks, which are rings of unidirectionally or bidirectionally coupled neurons, have been extensively investigated \cite{Pakdaman97,Campbell99,Campbell05,Campbell06,Guo07,Horikawa09a}. In terms of the number of inhibitory connections, unidirectional ring networks can be divided into even and odd networks in accordance to the parity of the number of inhibitory connections \cite{Pasemann95}. It has been well known that odd networks are capable of generating sustained oscillations, and numerical simulations showed that although the theoretical results suggested eventual convergence of any solution trajectory in even networks \cite{Smith95,Guo07,Pakdaman97}, long lasting oscillations can be very easily observed \cite{Pakdaman97}. Pakdaman et al. \cite{Pakdaman97} studied the asymptotic behavior of the system, and showed that the long lasting transient oscillations can not be explained by the analysis of the asymptotic behavior of the network. Horikawa et al and others \cite{Horikawa09a, Horikawa09b} studied different properties of the long lasting transient oscillations, such as the dependence of their durations on the number of neurons, effects of noise and variations on the durations of the long lasting transient oscillations in excitatory ring networks. Campbell and collaborators \cite{Yuan04,Campbell05,Campbell06} studied the stability and bifurcations of both the trivial solution and the nontrivial synchronous and asynchronous periodic solutions bifurcating from the trivial solution in ring networks of bidirectionally coupled neurons. For the general continuous-time Hopfield-type neural networks of $n$ neurons with different activation functions with and without delay, Cheng et al \cite{Shih06,Shih07} formulated the parameter conditions for the existence of the $2^n$ multiple stable stationary solutions, and estimated the basins of attraction of the coexisting multiple stable stationary solutions.

In \cite{Zhang13} we studied the construction of a general class of the Hopfield-type networks from arbitrarily preselected binary cyclic patterns representing phase-locked oscillations \cite{Chen01,Guo07}. We formulated and proved the admissibility condition for cycles under which a Hopfield-type network can be constructed using the pseudoinverse learning rule \cite{Personnaz86}. In terms of the structural features, we divided admissible cycles into three classes: simple cycles, separable and inseparable composite cycles, and showed that in terms of the structural features of admissible cycles, the topology of the networks constructed from them can be determined. Networks constructed from separable cycles consist of isolated clusters, and each cluster corresponds to a simple cycle component. Consequently, the dynamics of such networks is fully determined by that of the isolated clusters respectively.

In \cite{Zhang14}, we demonstrated that due to the misalignment of the zeros of the membrane potentials of the neurons in most networks constructed from admissible cycles, most admissible cycles are retrieved as the (delay-induced) long-lasting transient oscillations. To discriminate between the cycles that can be retrieved as attracting limit cycles and those that can be retrieved as the long-lasting transient oscillations, we introduced the concepts of strong retrievability and weak retrievability. In networks with the parameters $\lambda\rightarrow\infty$ and $C_0=0$, we proved that every admissible cycle with intermediate patterns satisfying its transition conditions is weakly retrievable \cite{Zhang14}.

In this paper, we consider the continuous-time Hopfield-type networks that are constructed from admissible binary cyclic patterns using the pseudoinverse learning rule, and study the retrievability of admissible cycles, as well as the structures of the local bifurcations of the trivial solutions. While retrieval of static binary patterns in both continuous-time and discrete-time Hopfield networks and their extensions have been extensively investigated (e.g. \cite{Agliari13,Tang10}), and studies involving storage and retrieval of oscillatory sequences in recurrent networks have been recently performed (e.g. \cite{Lai04,Tani08}), no systematic investigation of retrievability of cyclic patterns in networks with delay(s) have been done. In this paper, we follow \cite{Zhang14} to discuss retrievability of the admissible cycles in the networks with delayed couplings and more general parameter settings, i.e. $\lambda<\infty$ but still large and $0 < C_0 < 1$. We analyze the linear stability of the trivial solution, and show how the structural features of the prescribed cycle determine structure of the possible local bifurcations of the trivial solution. Accordingly, for each network constructed from an admissible cycle, the complete scenario of all possible local bifurcations of the trivial solution is obtained. With this scenario, we study structures of local bifurcations of the trivial solution and the periodic solutions bifurcating from the trivial solution using the MatLab packages MatCont 3.1 and DDE-BIFTOOL 2.03. We demonstrate that the cyclic patterns prescribed in the corresponding networks are stored and retrieved as different mathematical objects. Depending on their structural features, cyclic patterns are respectively stored and retrieved as attracting limit cycles, unstable periodic solutions bifurcating from the trivial solution, and the long lasting transient oscillations \cite{Pakdaman97,Horikawa09a,Horikawa09b}, which are suggested to be the consequence of the interactions among the unstable periodic solutions and stable/unstable steady state solutions.

The paper is organized as follows. In section 2, after a brief review the main results about admissibility of cycles, classification of admissible cycles, and network topology obtained in \cite{Zhang13}; and \emph{Misalignment Length Analysis} (MLA) method and the weak retrievability of a special class of admissible cycles presented in \cite{Zhang14}, we continue to prove that all admissible cycles satisfying the same transition conditions can be retrieved in the network constructed from the cycle imposing the transition condition and discuss retrievability of the admissible cycles in networks with more general parameter settings, i.e. $\lambda$ finite but sufficiently large and $0<C_0<1$. In section 3, we first present a detailed discussion on the boundaries of the stability region of the trivial solution, which enable us to determine all the possible codimension one local bifurcations of the trivial solution. Then with the numerical continuation and bifurcation analysis MatLab package, MatCont 3.1 and DDE-BIFTOOL 2.03, we continue to discuss structures of the local bifurcations of the trivial solution in networks constructed from admissible cycles with different structural features. We illustrate that admissible cycles are stored and retrieved in the networks with delay as different mathematical objects, attracting limit cycles, unstable periodic solutions and long lasting transient oscillations. In section 4, we discuss the significance and implications of our main results.

\section{Retrievability of Admissible Cycles}
\label{sec:Retrievability}
Before discussing the retrievability of admissible cycles in networks with more general parameter settings, we briefly review the main results obtained in \cite{Zhang13,Zhang14} to provide the necessary background, and prove one important corollary. 
\subsection{Hopfield-type networks and pseudoinverse learning rule}
\label{subsec:NetworkAndLearningRules}
In this paper, we consider the continuous-time Hopfield-type networks with delayed couplings of the following form (dots designate time derivatives)
\begin{equation}
	\label{eq:PotentialDelay00}
	\mathbf{\dot{u}}(t) = -\mathbf{u}(t) + C_0\beta_K\mathbf{J^0} \tanh(\lambda\mathbf{u}(t)) + C_1\beta_K\mathbf{J} \tanh(\lambda \mathbf{u}(t - \tau))
\end{equation}
where $\displaystyle{\beta_K = \frac{\mathrm{arctanh}(\beta_1)}{\lambda\beta_1}}$ and $\beta_1\in(0,1)$ is a parameter for incorporating the prescribed cycle of binary patterns (column vectors) into the network \cite{Gencic90,Zhang13}, $\mathbf{J^0}$ and $\mathbf{J}$ are network connectivity matrices for storing individual patterns in the prescribed cycle as fixed points and imposing the transitions between the memory states respectively, $C_0$ and $C_1 = 1 - C_0$ are the two parameters controlling the relative contributions of the two components of the network connectivities $\mathbf{J^0}$ and $\mathbf{J}$ respectively, and $\lambda$ is the gain scaling parameter \cite{Hopfield84}. The $i$-th component $u_i(t)$ of $\mathbf{u}(t) = (u_1(t)$, $u_2(t)$, $\dots$, $u_N(t))^T$ models the membrane potentials of the $i$-th neuron in the network, $\mathbf{I}$ is the $N\times N$ identity matrix, and $\tau\geq 0$ is the time delay. For the convenience of notations, here and subsequently, we follow \cite{Zhang13} and use the notations $f(\mathbf{x})$ $=$ $(f(x_1)$,$f(x_2)$,$\dots$,$f(x_N))^T$ and $\mathbf{x}^n$ $=$ $(x_1^n$, $x_2^n$, $\dots$, $x_N^n)$, where $f=\tanh$ or $\mathrm{arctanh}$, $\mathbf{x}$ $=$ $\lambda\mathbf{u}$ or $\mathbf{v}$ and $n\in\mathbb{N}$. Let $v_i = \tanh(\lambda u_i)$ model the firing rates of the $i$-th neuron $1\leq i\leq n$, then (\ref{eq:PotentialDelay00}) can be written as
\begin{equation}
	\label{eq:RateDelay00}
	\mathbf{\dot{v}}(t) = \lambda(\mathbf{I} - \mathrm{diag}(\mathbf{v}^2(t)))\Big{(}C_0\beta_K\mathbf{J^0}\mathbf{v}(t) + C_1\beta_K\mathbf{J}\mathbf{v}(t - \tau) - \frac{\mathrm{arctanh}(\mathbf{v}(t))}{\lambda}\Big{)}
\end{equation}

Following \cite{Personnaz86, Gencic90, Zhang13}, we denote a cycle of $p$ $N$-dimensional binary patterns by $\Sigma = (\xi^{(1)}$, $\xi^{(2)}$, $\dots$, $\xi^{(p)})$ and the same cycle after cyclically shifting leftwards by one pattern by $\mathbf{F} = (\xi^{(2)}$, $\xi^{(3)}$, $\dots$, $\xi^{(p)}$, $\xi^{(1)})$, where $\xi^{(\mu)}=(\xi_1^{(\mu)},\xi_2^{(\mu)},\dots,\xi_N^{(\mu)})^T$ with $\xi_i^{(\mu)} = 1$ or $-1$, for $1\leq\mu\leq p$ and $1\leq i\leq N$. Thus, in matrix formulation, we have $\mathbf{F} = \Sigma \mathbf{P}$, where
\begin{equation}
	\label{eq:P}
	\mathbf{P} = \left(\begin{array}{cccccc}
		   0   &    0   &    0   & \cdots &    0   &    1   \\
		   1   &    0   &    0   & \cdots &    0   &    0   \\
		   0   &    1   &    0   & \cdots &    0   &    0   \\
		\vdots & \vdots & \vdots & \ddots & \vdots & \vdots \\
		   0   &    0   &    0   & \cdots &    0   &    0   \\
		   0   &    0   &    0   & \cdots &    1   &    0
	\end{array}\right).
\end{equation}
Using the pseudoinverse learning rule \cite{Kohonen76,Personnaz86}, the two components of the network connectivities, $\mathbf{J^0}$ and $\mathbf{J}$, are respectively constructed as
\begin{equation}
	\label{eq:J0}
	\mathbf{J}^0 = \Sigma \Sigma^+
\end{equation}
and
\begin{equation}
	\label{eq:J1}
	\mathbf{J} = \mathbf{F} \Sigma^+
\end{equation}
where $\Sigma^+$ is the Moore-Penrose pseudoinverse of the matrix $\Sigma$. Personnaz et al. \cite{Personnaz86} referred to (\ref{eq:J0}) as the \emph{projection learning rule} and (\ref{eq:J1}) as the \emph{associating learning rule}. The component $\mathbf{J^0}$ serves to store every column vector in $\Sigma$, also called by us ``individual pattern'', as fixed points. If the patterns of $\Sigma$ are mutually orthogonal, the projection rule (\ref{eq:J0}) reduces to the Hebb's rule. The component $\mathbf{J}$ imposes the transitions among patterns prescribed by the cycle $\Sigma$. Personnaz et al. \cite{Personnaz86} pointed out that if $\mathbf{F} \Sigma^+ \Sigma = \mathbf{F}$, the matrix equation $\mathbf{J} \Sigma = \mathbf{F}$ has an exact solution $\mathbf{J} = \mathbf{F}\Sigma^+$.

A cycle $\Sigma$ is said to be \emph{admissible} if it can be stored in a network in the sense that a connectivity matrix $\mathbf{J}$ can be constructed from $\Sigma$ with the pseudoinverse learning rule (\ref{eq:J1}) such that $\mathbf{J}\Sigma = \mathbf{F}$. In \cite{Zhang13}, we have investigated storage of admissible cycles. It was proved that a cycle is admissible if and only if its discrete Fourier transform contains exactly $r = \mathrm{rank}(\Sigma)$ nonzero columns. In terms of the structural features, admissible cycles were classified into three categories, simple cycles, separable and inseparable composite cycles. The topology of the networks constructed from each type of the cycles was analyzed. In the next two subsections, we prove that admissibility implies weak retrievability in the networks with delayed couplings and $\lambda$ sufficiently large, and discuss retrieval of different types of admissible cycles.

\subsection{Retrievability of Admissible Cycles in the Networks with $C_0 = 0$}
\label{subsec:RetrievabilityC0=0}

\noindent\begin{definition}
	\label{def:SatisfyingTransitionCondition00}
		Let $\Sigma = (\xi^{(1)},\dots,\xi^{(p)})$ be an admissible cycle, and $\mathbf{u}(t)$ a solution of the network (\ref{eq:PotentialDelay00}) constructed from $\Sigma$ starting from an initial data $\varphi(\theta) = a\xi^{(\mu)}$ for all $\theta\in[-\tau,0]$, $\xi^{(\mu)}\in\Sigma$ and $a\in\mathbb{R}^+$. We say that $\mathbf{u}(t)$ \textbf{satisfies the transition conditions imposed by $\Sigma$} over an interval $(0,T)$, if the solution $\mathbf{u}(t)$ satisfies
	\begin{equation}
		\mathrm{sign}(\mathbf{u}(t)) = \xi^{((\mu+n+1)\mbox{ mod }p) + 1}\mbox{, for all }t\in\mathcal{I}_n.
		\label{eq:TransCondition}
	\end{equation}
	\label{def:TransCondtion}
\end{definition}

\begin{definition} 
	\label{def:StrongRetrievability}
		An admissible cycle $\Sigma$ is \textbf{strongly retrievable}, if the system (\ref{eq:PotentialDelay00}) has a solution $\mathbf{u}(t)$ such that $\mathbf{u}(t)$ satisfies the transition conditions imposed by $\Sigma$ on $(0,T)$ for all $T>0$.
\end{definition}

\begin{definition}
	An admissible cycle $\Sigma$ is \textbf{weakly retrievable}, if for any integer $m>0$, there exists a $\tau_0>0$, such that the largest number $\hat{n}$ of the prescribed patterns every network constructed from $\Sigma$ with $\tau\geq \tau_0$ is greater than $m$.
\end{definition}

	Next, we restate an important result on weak retrievability of a special class of admissible cycles proved in \cite{Zhang14} in the networks with $C_0 = 0$ and $\lambda\rightarrow\infty$ and prove an important consequence of it. For networks with $C_0 > 0$, the prescribed admissible cycles may or may not be retrieved in the networks constructed from them. We will discuss the retrieval of admissible cycles in such networks in examples in Section~\ref{subsec:RetrievabilityC0>0}. 

\begin{lemma}
	Let $\Delta T_n$ and $\Delta\tilde{T}_n$ be the misalignments of two networks constructed from admissible cycles. If $\Delta T_n \leq \Delta\tilde{T}_n$ for all $n$, and the network associate to $\Delta\tilde{T}_n$ is weakly retrievable, then the network associated to $\Delta T_n$ is weakly retrievable too.
	\label{lem:MisalignmentOrder}
\end{lemma}

\begin{theorem}
	\label{thm:RetrievabilityC0=0}
	Every simple MC-cycle $\Sigma$ with intermediate patterns satisfying the transition conditions imposed by $\Sigma$ is weakly retrievable in the network constructed from it with $\lambda$ sufficiently large.
\end{theorem}


\begin{corollary}
	\label{cor:CoexistingCycles00}
	All admissible cycles in the set $\{\Sigma_i\in\{-1,1\}^{N\times p}$ $:$ $\mathbf{J}\Sigma_i = \Sigma_i\mathbf{P}$ $\mathrm{with}$ $\mathbf{J}$ $\mathrm{fixed}\}$ can be retrieved in the same network with connectivity matrix $\mathbf{J}$.
\end{corollary}

\begin{proof}
	Noticing that in the proof of Theorem~\ref{thm:RetrievabilityC0=0}\cite{Zhang14}, although the network is constructed from $\Sigma$, to guarantee that a cycle $\Sigma_i$ is retrievable in the same network, $\Sigma_i$ $=$ $\Sigma$ is not required. Instead, satisfying the transition condition $\mathbf{J}\Sigma_i$ $=$ $\Sigma_i\mathbf{P}$ is the only requirement. Therefore, the assertion follows naturally.
\end{proof}

\begin{remark}{\rm
	For the convenience of discussion, we call the cycles that can be retrieved but are not prescribed in the network the \emph{derived cycles}.

	Although Theorem~\ref{thm:RetrievabilityC0=0} requires $\Sigma$ to be the cycle prescribed in the network, i.e., the network is constructed from it, it is important to notice that $\Sigma$ is not the only cycle that can be retrieved in the network. From Corollary~\ref{cor:CoexistingCycles00}, we have seen that any cycle satisfying the transition condition imposed by $\Sigma$ can be retrieved in the network constructed from $\Sigma$. In the next example, we demonstrate that in the network constructed from a randomly chosen prescribed admissible simple cycle \cite{Zhang13}
\begin{equation}
	\label{eq:PrescribedCycle00}
	\Sigma = \left(\begin{array}{cccccc}
		+ & + & - & + & - & - \\
		+ & - & + & - & - & + \\
		- & + & - & - & + & + \\
		+ & - & - & + & + & - \\
		- & - & + & + & - & +
	\end{array}\right),
\end{equation}	
	in addition to the prescribed cycle (\ref{eq:PrescribedCycle00}), three other cycles can be successfully retrieved too.
	}
	\label{rem:Remark4}
\end{remark}

\begin{example}
	\label{ex:Example01}{\rm 
	Consider the network constructed from the cycle (\ref{eq:PrescribedCycle00}) using the pseudoinverse learning rule. The connectivity matrices $\mathbf{J^0}$ and $\mathbf{J}$ are respectively constructed as  $\mathbf{J^0} = \Sigma\Sigma^+$ and $\mathbf{J} = \Sigma\mathbf{P}\Sigma^+$. Since if a cycle is retrievable in a network, the cycle itself has to satisfy the transition conditions prescribed in the network, to reveal how many cycles satisfy the transition conditions imposed by the cycle $\Sigma$, we investigate the evolutions of all binary state patterns $\xi\in\{-1,1\}^5$ under the transition operation $\xi\mapsto\mathrm{sgn}(\mathbf{J}\xi)$. In Figure~\ref{fig:Figure4}, we illustrate the evolution graphs of the binary state patterns. Following \cite{Personnaz86}, we represent each binary column vectors $\xi$ $\in$ $\{-1,1\}^N$ by a decimal integer. Before converting binary vectors to decimal integers, the digit $-1$ in $\xi$ is replaced by $0$. For instance, $\xi$ $=$ $(+,+,-,+,-)^T$ is replaced by $(1,1,0,1,0)$ first, then converted to $2^4$ $+$ $2^3$ $+$ $2^1$ $=$ $16$ $+$ $8$ $+$ $2$ $=$ $26$. 

	The evolution graphs exhibit four loops. Direct calculations show that these four loops correspond to four cycles satisfying the transition conditions imposed by the prescribed cycle $\Sigma$. Labeling these four cycles by $\Sigma_1$, $\Sigma_2$, $\Sigma_3$ and $\Sigma_4$ respectively (see Figure~\ref{fig:Figure4}) yields the four equalities, $\mathbf{J}\Sigma_i = \Sigma_i\mathbf{P}$, $i=1,2,3,4$, where $\mathbf{P}$ is the cyclic permutation matrix of the form (\ref{eq:P}). For $i\not=4$, $\mathbf{P}$ is of order $6$, and for $i=4$, $\mathbf{P}$ is of order $2$. Clearly, the cycle $\Sigma_3$ is the prescribed cycle $\Sigma$.

	Figure~\ref{fig:Figure5}\textbf{A} and \textbf{B} respectively illustrate the retrieved time series in raster plots of the four cycles and the projections of their phase trajectories onto the three-dimensional $(u_2,u_3,u_4)$ phase subspace. Both illustrations suggest that the four cycles are retrieved successfully. This confirms Corollary~\ref{cor:CoexistingCycles00}.
	}
\end{example}

\begin{figure}
	\begin{center}
		\includegraphics[height=2.5in]{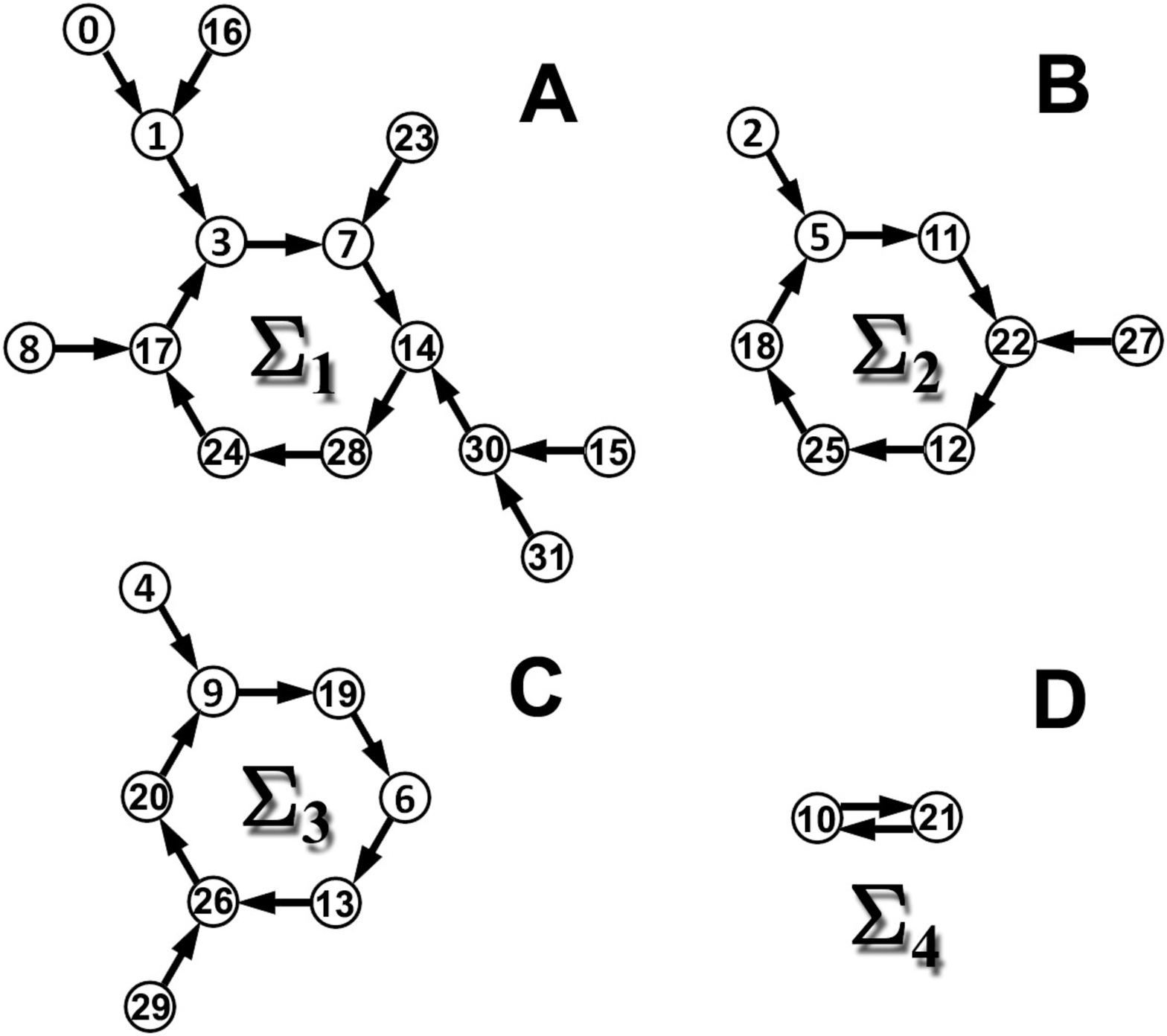}
	\end{center}
	\caption{\footnotesize{Graphs of evolution of the state patterns $\xi^{(\mu)}$ in the state space $\{-1,1\}^5$ under the transition operation $\xi^{(\mu)}\mapsto\mathrm{sgn}(\mathbf{J}\xi^{(\mu)})$, where $\mathbf{J}$ is constructed from the cycle (\ref{eq:PrescribedCycle00}) using the Pseudoinverse learning rule.  (details refer to the text).}}
	\label{fig:Figure4}
\end{figure}

\begin{figure}
	\begin{center}
		\includegraphics[height=2.5in]{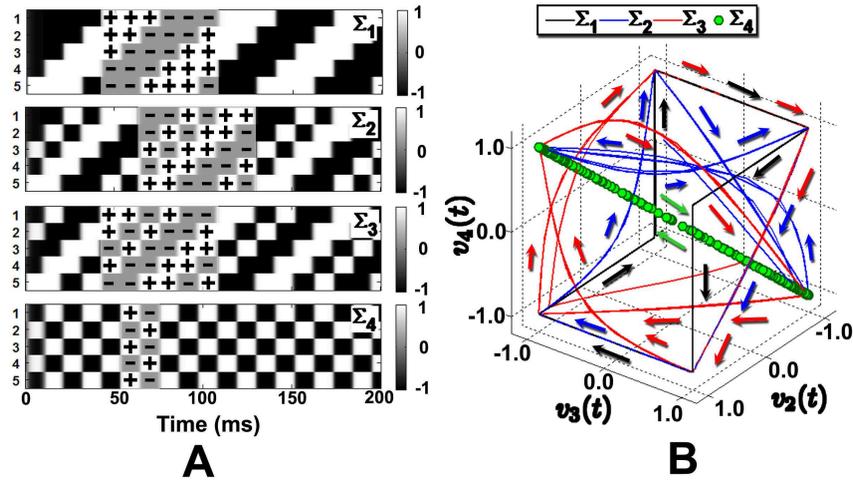}
	\end{center}
	\caption{\footnotesize{Retrieval of the 4 coexisting derived cycles. \textbf{A} Raster plots of the time series of the four retrieved cycles, $\Sigma_1$, $\Sigma_2$, $\Sigma_3$, and $\Sigma_4$. The network is constructed from the cycle (\ref{eq:PrescribedCycle00}). Clearly, the prescribed cycle $\Sigma = \Sigma_3 = -\Sigma_2$, and $\Sigma_1$ $=$ $(\hat{\eta}^T$, $(\hat{\eta}\mathbf{P})^T$, $\dots$, $(\hat{\eta}\mathbf{P}^5)^T)^T$, with $\hat{\eta} = (+,+,+,-,-,-)$. The cycle $\Sigma_4$ is of length two and $\mathbf{J}\Sigma_4 = \Sigma_4\mathbf{P} = -\Sigma_4$. The $i$-th row in each raster plot represents the firing rates $v_i(t)$ of the $i$-th neuron. The parameters are set as $\beta=3$, $C_0 = 0$, $\tau = 10\mathrm{ms}$, and $\lambda = 20$. \textbf{B} Projections of the phase trajectories of the four cycles onto the $(u_2,u_3,u_4)$ phase subspace.}}
	\label{fig:Figure5}
\end{figure}

In next subsection, we discuss retrieval of prescribed and derived admissible cycles in the networks with $C_0>0$.

\subsection{Retrievability of Admissible Cycles in the Networks with $C_0 > 0$}
\label{subsec:RetrievabilityC0>0}
In general, the connectivity of a network (\ref{eq:PotentialDelay00}) contains two components \cite{Gencic90,Zhang13}, $\mathbf{J^0}$ and $\mathbf{J}$. One is for stabilizing individual patterns in $\Sigma$ as fixed points of the system, and the other is for imposing the transitions among patterns prescribed by the cycle $\Sigma$. We use the two parameters $C_0$ and $C_1 = 1 - C_0$ to control the relative contributions of these two components in shaping the dynamics of the network. In Section 2.2, we proved in Theorem \ref{thm:RetrievabilityC0=0} that when only the transition component $\mathbf{J}$ is included, i.e. $C_0 = 0$, every prescribed cycle can be retrieved with both $\lambda$ and $\tau$ sufficiently large. For $C_0 = 1$, transition conditions imposed by the prescribed cycle $\Sigma$ are completely removed. Consequently, no cycle is stored in such networks \cite{Personnaz86}, and we will not consider this type of networks here. In this subsection, we consider networks with $0<C_0<1$, and discuss retrieval of the prescribed cycles.

\begin{theorem}
	\label{thm:RetrievabilityC0>0}
	Suppose a network of the form (\ref{eq:PotentialDelay00}) is constructed from an admissible cycle $\Sigma$ $=$ $(\xi^{(1)}$, $\xi^{(2)}$, $\dots$, $\xi^{(p)})$, and in the $(n+1)$-th time interval $[n\tau,(n+1)\tau)$, the solution of the network (\ref{eq:PotentialDelay00}) satisfies $\mathrm{sgn}(\mathbf{u}(t)) = \xi^{(\mu)}$ for some $1\leq\mu\leq p$, then in the next time interval $t\in[(n+1)\tau,(n+2)\tau)$, $\beta_K\beta_1\xi^{(\mu+1)}$ is an asymptotically stable equilibrium point of (\ref{eq:PotentialDelay00}).
\end{theorem}

\begin{proof}
	Since $\mathrm{sgn}(\mathbf{u}(t)) = \xi^{(\mu)}$ for all $t\in[n\tau,(n+1)\tau)$, following from Lemma 1, we may assume $\mathbf{v}(t)$ $=$ $\beta_1\xi^{(\mu)}$, or equivalently, $\mathbf{u}(t)$ $=$ $\mathrm{arctanh}(\beta_1\xi^{(\mu)})/\lambda$ $=$ $\beta_K\beta_1\xi^{(\mu)}$. Thus, constraining ourselves in the $(n+2)$-th time interval, i.e., $t\in[(n+1)\tau,(n+2)\tau)$, and substituting $\mathbf{u}(t-\tau)$ $=$ $\beta_K\beta_1\xi^{(\mu)}$ into the network (\ref{eq:PotentialDelay00}), we obtain a nonlinear system of ordinary different equations
	$$
		\mathbf{\dot{u}}(t) = -\mathbf{u}(t) + C_0\beta_K\mathbf{J^0}\tanh(\lambda\mathbf{u}(t)) + C_1\beta_K\mathbf{J}\beta_1\xi^{(\mu)}.
	$$
	Since $\mathbf{J}\xi^{(\mu)} = \xi^{(\mu+1)}$ \cite{Gencic90,Zhang13}, we have 
	\begin{equation}
		\label{eq:RetrieveODE00}
		\mathbf{\dot{u}}(t) = -\mathbf{u}(t) + C_0\beta_K\mathbf{J^0}\tanh(\lambda\mathbf{u}(t)) + C_1\beta_K\beta_1\xi^{(\mu+1)}.
	\end{equation}
	Suppose $\mathbf{u^*}$ is an equilibrium solution of (\ref{eq:RetrieveODE00}); substituting $\mathbf{u}(t) = \mathbf{u^*}$ into (\ref{eq:RetrieveODE00}) gives
	$$
		\mathbf{u^*} = C_0\beta_K\mathbf{J^0}\tanh(\lambda\mathbf{u^*}) + C_1\beta_K\beta_1\xi^{(\mu+1)}
	$$
	Let $\mathbf{v^*} = \tanh(\lambda\mathbf{u^*})$, then
	$$
		\frac{1}{\lambda}\mathrm{arctanh}(\mathbf{v^*}) = C_0\beta_K\mathbf{J^0}\mathbf{v^*} + C_1\beta_K\beta_1\xi^{(\mu+1)}.
	$$
	Direct substitution shows that $\mathbf{v^*} = \beta_1\xi^{(\mu+1)}$ is a solution of the above equations. This shows that $\mathbf{u^*}$ $=$ $\beta_K\beta_1\xi^{(\mu+1)}$ is an equilibrium solution of (\ref{eq:RetrieveODE00}). 
	
	Next, we show that the equilibrium solution $\mathbf{u^*} = \beta_K\beta_1\xi^{(\mu+1)}$ is asymptotically stable on $[(n+1)\tau,(n+2)\tau)$. Let $\mathbf{u}(t)$ be a perturbed solution around the equilibrium solution $\mathbf{u^*}$, that is, $\mathbf{u}(t) = \mathbf{u^*} + \delta\mathbf{u}(t)$. Thus, $\mathbf{\dot{u}}(t) = \delta\mathbf{\dot{u}}(t)$. Linearizing the right hand side of the inhomogeneous system (\ref{eq:RetrieveODE00}) around $\mathbf{u^*}$ and substituting $\mathbf{u^*} = \beta_K\beta_1\xi^{(\mu+1)}$ into the linearized system yield
	$$
		\delta\mathbf{\dot{u}}(t) = C_0\beta\mathbf{J^0}\Xi\delta\mathbf{u}(t) - \delta\mathbf{u}(t),
	$$
	where $\beta$ $=$ $\beta_K\lambda$, and $\Xi$ $=$ $\mathrm{diag}(1$ $-$ $\tanh^2(\lambda\beta_K\beta_1\xi_1^{(\mu+1)})$, $\dots$, $1$ $-$ $\tanh^2(\lambda$ $\beta_K\beta_1\xi_N^{(\mu+1)}))$. Since $\xi_i^{(\mu+1)}$ $\in$ $\{-1,1\}$ for every $i$, and $\lambda\beta_K\beta_1$ $=$ $\mathrm{arctanh}(\beta_1)$, it follows that $\Xi$ $=$ $(1 - \beta_1^2)\mathbf{I}$. Therefore, we get
	\begin{equation}
		\label{eq:LinearizedRetrieveODE00}
		\delta\mathbf{\dot{u}}(t) = \mathbf{A}\delta\mathbf{u}(t),
	\end{equation}
	where $\mathbf{A} = C_0\beta(1-\beta_1^2)\mathbf{J^0} - \mathbf{I}$. Since the connectivity matrix $\mathbf{J^0}$ and the identity matrix $\mathbf{I}$ commute, they are simultaneously diagonalizable. That is, if $\mathbf{Q}$ diagonalizes $\mathbf{A}$, i.e. $\mathbf{Q}^{-1}\mathbf{A}\mathbf{Q}$ $=$ $\Lambda$ $=$ $\mathrm{diag}(\sigma_1$, $\sigma_2$, $\dots$, $\sigma_N)$, with $\Re(\sigma_1)$ $\geq$ $\Re(\sigma_2)$ $\geq$ $\cdots$ $\geq$ $\Re(\sigma_N)$, where $\Re(x)$ designates the real part of the complex number $x$, then 
	\begin{equation}
		\label{eq:ALambda00}
		\Lambda = C_0\beta(1-\beta_1^2)\mathbf{Q}^{-1}\mathbf{J^0}\mathbf{Q} - \mathbf{I}.
	\end{equation}
	Since $\mathbf{J^0} = \Sigma\Sigma^+$ is idempotent, the eigenvalues of $\mathbf{J^0}$ are either 0 or 1. Thus, 
	$$
		\sigma_i = C_0\beta(1-\beta_1^2) - 1\mbox{ or }-1
	$$
	Since $\displaystyle{\beta = \frac{\mathrm{arctanh}(\beta_1)}{\beta_1}}$, and $0<\beta_1<1$, it follows that
	\begin{equation}
		\label{eq:RetrieveEigenvalues00}
		\sigma_i = C_0\frac{\mathrm{arctanh}(\beta_1)(1-\beta_1^2)}{\beta_1} - 1\mbox{ or }-1
	\end{equation}
	Since $C_0\leq 1$ and $\displaystyle{\frac{\mathrm{arctanh}(\beta_1)(1-\beta_1^2)}{\beta_1}}$ is a monotonically decreasing function of $\beta_1$ over the open interval $(0,1)$ and 
	$$
		\lim\limits_{\beta_1\rightarrow 0}\frac{\mathrm{arctanh}(\beta_1)(1-\beta_1^2)}{\beta_1} = 1
	$$
	and
	$$
		\lim\limits_{\beta_1\rightarrow 1}\frac{\mathrm{arctanh}(\beta_1)(1-\beta_1^2)}{\beta_1} = 0,
	$$
	it follows that $\Re(\lambda_i)<0$ for all $i$. This shows that $\mathbf{u^*} = \beta_K\beta_1\xi^{(\mu+1)}$ is asymptotically stable over the interval $t\in[(n+1)\tau,(n+2)\tau)$, hence completes the proof.
\end{proof}

	\bigskip

	Suppose the pattern $\xi^{(\mu)}$ has been retrieved successfully in the $(n+1)$-th time interval, i.e. $\mathrm{sgn}(\mathbf{u}(t)) = \xi^{(\mu)}$ for $t\in[n\tau,(n+1)\tau)$ neglecting the transient transitions. Substitute $\mathbf{u}(t-\tau) = \beta_K\beta_1\xi^{(\mu)}$ into (\ref{eq:PotentialDelay00}) and consider the continuity of the solution at $t$ $=$ $(n+1)\tau$, we obtain the following derived initial value problem as follows
	\begin{equation}
		\label{eq:IVPODE00}
			\begin{cases}
				\begin{array}{ccl}
					\mathbf{\dot{u}}(t)   & = & -\mathbf{u}(t) + C_0\beta_K\mathbf{J^0}\tanh(\lambda\mathbf{u}(t)) + C_1\beta_K\beta_1\xi^{(\mu+1)} \\
					\mathbf{u}((n+1)\tau) & = & \beta_K\beta_1\xi^{(\mu)}
				\end{array}
			\end{cases}.
	\end{equation}
	Theorem \ref{thm:RetrievabilityC0>0} guarantees that in the next time interval, $\mathbf{u^*} = \beta_K\beta_1\xi^{(\mu+1)}$ is an asymptotically stable equilibrium point. Therefore, if in the $(n+1)$-th time interval, the solution $\mathbf{u}(t)$ entered the basin of attraction of $\mathbf{u^*}$ in the $(n+2)$-th time interval, then $\mathbf{u}(t)$ would be attracted to $\mathbf{u}^*$ in the $(n+2)$-th time interval. Accordingly, the next pattern $\xi^{(\mu+1)}$ in the cycle $\Sigma$ would be retrieved successfully. Next we follow Cheng et al. \cite{Shih06,Shih07} to adopt a similar geometric-based method to analyze the basin of attraction of the stable equilibria of the above derived nonlinear system (\ref{eq:IVPODE00}). 
	
	Let the right-hand side of (\ref{eq:IVPODE00}) be denoted by $\mathbf{f}(\mathbf{u}) = (f_1(\mathbf{u}),\dots,f_N(\mathbf{u}))^T$, where for every $i=1,2,\dots,N$,
	\begin{equation}
		\label{eq:RHSODE00}
			f_i(\mathbf{u}) = -u_i + C_0\beta_K\sum\limits_{j=1}^N J_{ij}^0\tanh(\lambda u_j) + C_1\beta_K\beta_1\xi_i^{(\mu+1)}.
	\end{equation}
	
	Next, we prove two simple but useful results.
	
	\begin{lemma}
		Let $\mathbf{J^0}$ be constructed from any admissible cycle $\Sigma\in\{-1,1\}^{N\times p}$ using the pseudoinverse learning rule (\ref{eq:J0}). Then $J_{ii}^0\geq 0$ for every $i=1,2,\dots,N$.
		\label{lem:J0iiNonNegative00}
	\end{lemma}
	\begin{proof}
		Let $\Sigma = U\Lambda V^*$ be the singular value decomposition of $\Sigma$, where $\Lambda$ is an $N\times p$ diagonal matrix, and $V^*$ is the adjoint of $V$. Then $\Sigma^+ = V\Lambda^+ U^*$, where $\Lambda^+$ is the $p\times N$ diagonal matrix obtained from $\Lambda$ by taking the reciprocal of each non-zero element on the diagonal of $\Lambda$, leaving the zeros in place, and transposing the resulting matrix. Thus, if $\Lambda$ $=$ $\mathrm{diag}(\sigma_1$, $\sigma_2$, $\dots$, $\sigma_k$, $0$, $\dots$, $0)$ with $\sigma_i\not=0$ for every $i\leq k$, and $k\leq\min\{N,p\}$, then $\Lambda^+$ $=$ $\mathrm{diag}(\sigma_1^{-1}$, $\sigma_2^{-1}$, $\dots$, $\sigma_k^{-1}$, $0$, $\dots$, $0)$ and $\Lambda\Lambda^+$ $=$ $\mathrm{diag}(1$, $1$, $\dots$, $1$, $0$, $\dots$, $0)$. Since $\mathbf{J^0}$ $=$ $\Sigma\Sigma^+$ $=$ $U\Lambda\Lambda^+ U^*$, it follows that $J_{ii}^0$ $=$ $\sum\limits_{j=1}^k U_{ij}U_{ji}^*$, where $U_{ij}$ and $U_{ji}^*$ are the entries in the place $(i,j)$ and $(j,i)$ in the matrices $U$ and $U^*$ respectively. Since $U_{ji}^*$ $=$ $(U_{ij})^*$, it follows that $U_{ij}U_{ji}^*$ $=$ $\|U_{ij}\|^2$, where $\|U_{ij}\|$ is the modulus of the complex number $U_{ij}$. Therefore, we have that
		$$
			J_{ii}^0 = \sum\limits_{j=1}^k \|U_{ij}\|^2.
		$$
		Since $\|U_{ij}\|^2 \geq 0$ for every $j\leq k$, it follows that $J_{ii}^0 \geq 0$ for every $i$.
	\end{proof}
	
	\begin{lemma}
		\label{lem:TwoTurningPoints00}
		Let $i=1,\dots,N$ be fixed. If $C_0\beta J_{ii}^0 > 1$, then there exist two values $p_i$ and $q_i$ of the $i$-th component $u_i$ of the membrane potential function $\mathbf{u}$ with $p_i < 0 < q_i$, such that $\partial_i f_i(\mathbf{u})|_{u_i = p_i} = 0$ and $\partial_i f_i(\mathbf{u})|_{u_i = q_i} = 0$ for $i=1,\dots,N$.
	\end{lemma}
	
	\begin{proof}
		Since $\partial_i f_i(\mathbf{u}) = C_0\beta_K J_{ii}^0\lambda(1 - \tanh^2(\lambda u_i)) - 1$, setting $\partial_i f_i(\mathbf{u}) = 0$ and replacing $\beta_K\lambda$ by $\beta$ gives $C_0\beta J_{ii}^0(1 - \tanh^2(\lambda u_i)) = 1$. It follows that
		\begin{equation}
			\label{eq:TwoSolutionCondition00}
			\tanh^2(\lambda u_i) = \frac{C_0\beta J_{ii}^0 - 1}{C_0\beta J_{ii}^0},
		\end{equation}
		which implies that the equation $\partial_i f_i(\mathbf{u}) = 0$ has two distinct solutions $p_i$ and $q_i$ with $p_i < 0 < q_i$ only when $C_0\beta J_{ii}^0 > 1$. This completes the proof.
	\end{proof}
	
	\begin{remark}{\rm
		The last equality (\ref{eq:TwoSolutionCondition00}) implies that if $C_0\beta J_{ii}^0 > 1$, then there exist two distinct real values $p_i$ and $q_i$ with $p_i < 0 < q_i$ such that $\partial_i f_i(\mathbf{u})|_{u_i = p_i} = 0$ and  $\partial_i f_i(\mathbf{u})|_{u_i = q_i} = 0$; if $C_0\beta J_{ii}^0 = 1$, then $\partial_i f_i(\mathbf{u}) = 0$ only at $u_i = 0$; if $C_0\beta J_{ii}^0 < 1$, then $\partial_i f_i(\mathbf{u}) < 0$ for all $u_i\in\mathbb{R}$. Therefore, if $C_0\beta J_{ii}^0 < 1$ for every $i$, then the equation for equilibrium points $\mathbf{f}(\mathbf{u}) = \mathbf{0}$ will have only one solution. Direct substitution of $\mathbf{u} = \beta_K\beta_1\xi^{(\mu+1)}$ into the equation $\mathbf{f}(\mathbf{u}) = \mathbf{0}$ shows that $\beta_K\beta_1\xi^{(\mu+1)}$ is the solution. Since in this case $\partial_i f_i(\mathbf{u}) < 0$ for all $i$ and $\mathbf{u}\in\mathbb{R}^N$, it follows that the unique equilibrium $\mathbf{u^*} = \beta_K\beta_1\xi^{(\mu+1)}$ is globally asymptotically stable. Accordingly, we have the that, if $C_0\beta J_{ii}^0 < 1$ and $\tau$ is sufficiently large in comparison to the time span of the transients, then any admissible cycle is retrievable, and any cycle satisfying the transition condition imposed by the prescribed cycle can be retrieved in the network constructed from the prescribed cycle.
	}\label{rem:Remark5}
	\end{remark}
	
	\bigskip

		

\begin{figure}
	\begin{center}
		\includegraphics[height=2.2in]{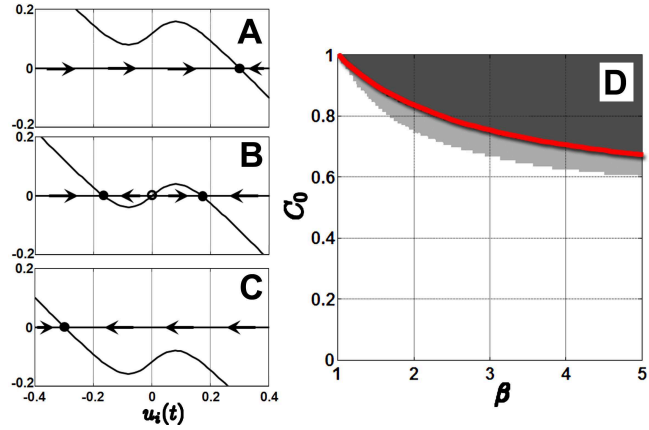}
	\end{center}
	\caption{\footnotesize{The graphs for $\hat{f}_i(u_i)$ (\textbf{A}), $\bar{f}_i(u_i)$ (\textbf{B}), and $\check{f}_i(u_i)$ (\textbf{C}). The parameters for the three curves are set as follows, $C_0 = 0.6$, $\beta = 3$, $\lambda=10$, $J_{ii}^0 = 1$ and $J_{ij}^0 = 0$ for all $i\not=j$ (this corresponds to $\Sigma$ being separable, minimal and consecutive \cite{Zhang13}). The red curve in \textbf{D} is the solution curve of the equation (\ref{eq:SNBifurCurve00}). The white, light gray and dark gray backgrounds in \textbf{D} correspond to the regions in the parameter plane where the system (\ref{eq:PotentialDelay00}) with ring topology with one inhibitory connection has $1$, $3$ or $1$, and $3^N$ equilibria respectively. On the boundary between light-gray region and dark-gray region, which coincides with the solution curve (red curve in \textbf{D}) of (\ref{eq:SNBifurCurve00}), both the multiple saddle-nodes on limit cycle bifurcations and saddle-node bifurcations occur. The former bifurcation breaks the limit cycle, and both bifurcations create all the rest equilibria.}}
	\label{fig:Figure6}
\end{figure}

	For $C_0\beta J_{ii}^0 > 1$, since $\xi^{(\mu)}\in\{-1,1\}^N$ for all $\mu\in\mathbb{N}$, and the ``forcing term'' $C_1\beta_K\beta_1\xi_i^{(\mu + 1)}$ in (\ref{eq:RHSODE00}) vertically shifts the curve of the function (see Figure~\ref{fig:Figure6}\textbf{A}, \textbf{B} and \textbf{C})
	\begin{equation}
		\label{eq:fBar00}
		\bar{f}_i(u_i) = -u_i + C_0\beta_K J_{ii}^0 \tanh(\lambda u_i),
	\end{equation}
	we have that $\check{f}_i(u_i)$ $\leq$ $f_i(\mathbf{u})$ $\leq$ $\hat{f}_i(u_i)$ for every $i=1,2,\dots,N$, where $\check{f}_i(u_i)$ $=$ $-u_i$ $+$ $C_0\beta_K J_{ii}^0$ $\tanh(\lambda u_i)$ $+$ $k_i^-$, $\hat{f}_i(u_i)$ $=$ $-u_i$ $+$ $C_0\beta_K J_{ii}^0$ $\tanh(\lambda u_i)$ $+$ $k_i^+$ and $k_i^-$ $=$ $-C_0\beta_K\displaystyle{\sum\limits_{j=1,j\not=i}^N} |J_{ij}^0|$ $-$ $C_1\beta_K\beta_1$, $k_i^+$ $=$ $C_0\beta_K$ $\displaystyle{\sum\limits_{j=1,j\not=i}^N}$ $|J_{ij}^0|$ $+$ $C_1\beta_K\beta_1$.
	
	If $\hat{f}_i(p_i)<0$ and $\check{f}_i(q_i)>0$, then the function $f_i(\mathbf{u})$ must intersect with the horizontal axis at three distinct points. Accordingly, if $\hat{f}_i(p_i)<0$ and $\check{f}_i(q_i)>0$ for every $i$, then the derived system in (\ref{eq:IVPODE00}) has $3^N$ equilibria. Moreover, with further constraints on parameters, $2^N$ of these $3^N$ equilibria become asymptotically stable. A similar result has been proven in \cite{Shih06}, and using the same arguments as in \cite{Shih06}, we can prove the following two lemmas. For the sake of the presentation, we summarize the three parameter conditions as ($\mathbf{H_1}$): $C_0\beta J_{ii}^0 > 1$; ($\mathbf{H_2}$): $\check{f}_i(q_i) > 0$ and $\hat{f}_i(p_i) < 0$; ($\mathbf{H_3}$): $\displaystyle{C_0\beta\sum\limits_{j=1}^N |J_{ij}^0|(1-\tanh^2(\lambda\eta_j)) < 1}$, where $\eta_j$ is chosen such that $\tanh^2(\lambda\eta_j)$ $=$ $\min\{\tanh^2(\lambda u_j) | u_j = \check{c}_j,\hat{a}_j\}$ with $\check{c}_j$ and $\hat{a}_j$ defined exactly same as \cite{Shih06}, and $i = 1,2,\dots,N$.
	
	\begin{lemma}
		Under ($\mathbf{H_1}$) and ($\mathbf{H_2}$), the derived nonlinear system in (\ref{eq:IVPODE00}) has $3^N$ equilibria.
		\label{lem:TotalNumberOfEquilibria00}
	\end{lemma}
	
	\begin{lemma}
		Under ($\mathbf{H_1}$), ($\mathbf{H_2}$) and ($\mathbf{H_3}$), the derived nonlinear system in (\ref{eq:IVPODE00}) has $2^N$ asymptotically stable equilibria.
	\end{lemma}

	Clearly, $\hat{f}_i(u_i)$ and $\check{f}_i(u_i)$ are respectively the upper and lower bounds for the function $f_i(\mathbf{u})$. Depending on the values of $C_0$ and $\beta$, the curve $f_i(\mathbf{u})$ may intersect with the $u_i(t)$ axis in one, two or three points. For $C_0\beta J_{ii}^0 > 0$, let $u_i = p_i$ and take $\xi_i^{(\mu+1)} = 1$, then setting $f_i(\mathbf{u}) = 0$ and substituting $\displaystyle{p_i = -\frac{1}{\lambda}\mathrm{arctanh}\left(\sqrt{\frac{C_0\beta - 1}{C_0\beta}}\right)}$ yield the following transcendental equation for the curve of the saddle-node bifurcations of the system (\ref{eq:IVPODE00}) occuring at $u_i = p_i$,
	\begin{equation}
		\label{eq:SNBifurCurve00}
		\mathrm{arctanh}\left(\sqrt{\frac{C_0\beta - 1}{C_0\beta}}\right) - \sqrt{C_0\beta(C_0\beta - 1)} + C_1\mathrm{arctanh}(\beta_1) = 0,
	\end{equation}
	where $\beta = \mathrm{arctanh}(\beta_1)/\beta_1$, and $\beta_1\in(0,1)$. For the saddle-node bifurcations of the derived system in (\ref{eq:IVPODE00}) occuring at $u_i = q_i$, the equation for the bifurcation curve can be obtained similarly by taking $u_i$ $=$ $q_i$ $=$ $-p_i$ and $\xi_i^{(\mu+1)} = -1$. The simple direct calculation shows that the equation is exactly same as (\ref{eq:SNBifurCurve00}). Therefore, we use it for finding the parameter values of both saddle-node bifurcations.
	
	We numerically solve the above equation (\ref{eq:SNBifurCurve00}) for $C_0$ with values of $\beta$ between 1 and 5. The red curve in Figure~\ref{fig:Figure6}\textbf{D} illustrates the solutions $(\beta,C_0)$ for $\beta\in(1,5]$. It is very interesting to notice that numerically this curve coincides with the curves of the multiple saddle-nodes on limit cycle bifurcations in the networks (\ref{eq:PotentialDelay00}) with ring topology with one inhibitory connection and without transmission time delay. In Figure~\ref{fig:Figure6}\textbf{D}, we use the background with different grayscales to indicate the regions in the $\beta$-$C_0$ parameter plane in which the corresponding networks (\ref{eq:PotentialDelay00}) with ring topology with one inhibitory connection have $1$ (white background), $3$ or $1$ (light gray backgound), and $3^N$ (dark gray backgound) equilibria. As we will see in the next section, a \emph{multiple saddle-nodes on limit cycle bifurcation} \cite{Hoppensteadt97} and several \emph{saddle-note bifurcations} occur on the boundary between the light-gray and dark-gray regions. The multiple saddle-nodes on the limit cycle bifurcation breaks the limit cycle, and both bifurcations create all of the rest equilibria in the networks without transmission delay. We will discuss this in more details in Section 3. In the next example we show that it is the saddle-node bifurcation in the system (\ref{eq:IVPODE00}) that breaks the successful retrievals of the prescribed cycles, and we also compare this bifurcation with the multiple saddle-nodes on limit cycle bifurcation in the corresponding network without transmission delay.

\begin{figure}
	\begin{center}
		\includegraphics[height=3in]{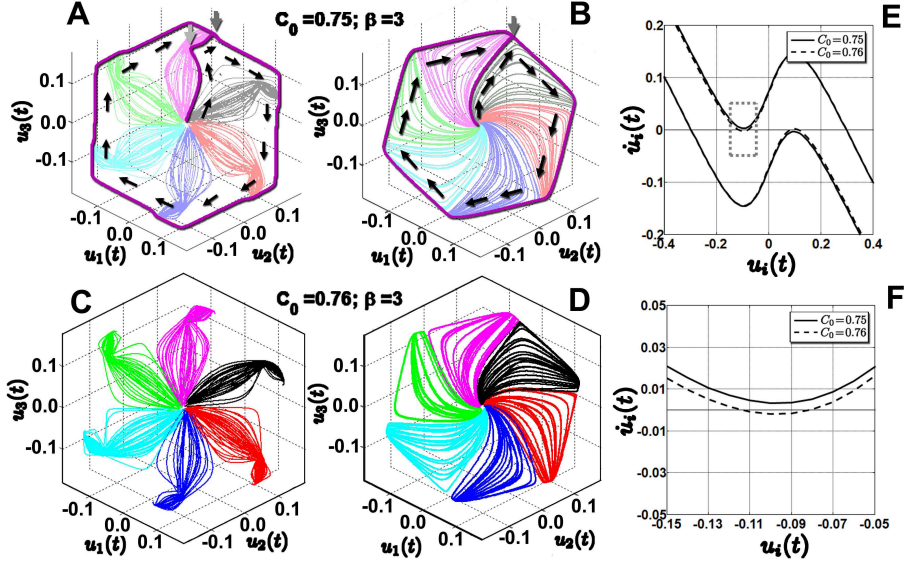}
	\end{center}
	\caption{\footnotesize{Breaking the cycles in the networks with (\textbf{A} and \textbf{C}) and without delay (\textbf{B} and \textbf{D}). The curves in \textbf{E} and \textbf{F} illustrate the occurrence of the saddle-node bifurcation in the derived system in (\ref{eq:IVPODE00}).}}
	\label{fig:Figure7}
\end{figure}

\begin{example}{\rm 
In order to compare the saddle-node bifurcation of the derived system in (\ref{eq:IVPODE00}) with the multiple saddle-nodes on limit cycle bifurcation in the corresponding network without transmission time delay, we consider the networks constructed using the pseudoinverse learning rule from the following admissible cycle,
\begin{equation}
	\label{eq:PrescribedCycle01}
	\Sigma = \left(\begin{array}{cccccc}
		+ & + & + & - & - & - \\
		+ & + & - & - & - & + \\
		+ & - & - & - & + & +
	\end{array}\right).
\end{equation}
The network is constructed as follows
\begin{equation}
	\label{eq:N3P6Delay00}
	\begin{cases}
		\begin{array}{ccl}
			\dot{u}_1(t) = -u_1(t) + C_0\beta_K\tanh(\lambda u_1(t)) + C_1\beta_K\tanh(\lambda u_2(t-\tau)) \\
			\dot{u}_2(t) = -u_2(t) + C_0\beta_K\tanh(\lambda u_2(t)) + C_1\beta_K\tanh(\lambda u_3(t-\tau)) \\
			\dot{u}_3(t) = -u_3(t) + C_0\beta_K\tanh(\lambda u_3(t)) - C_1\beta_K\tanh(\lambda u_1(t-\tau))
		\end{array}
	\end{cases}.
\end{equation}
It is easy to verify that only two cycles are stored in this network, one is the prescribed cycle (\ref{eq:PrescribedCycle01}), the other is the cycle of the two patterns, $(-,+,-)^T$ and $(+,-,+)^T$. In Section 3, we will show in the numerical continuation computations that in this network the prescribed cycle $\Sigma$ is stored as an attracting limit cycle. In Figure~\ref{fig:Figure7}\textbf{A}, \textbf{B}, \textbf{C} and \textbf{D}, we respectively illustrate 180 phase trajectories starting from the randomly chosen initial data which are very close to the origin in simulations in the networks with (\textbf{A}, \textbf{C}) and without (\textbf{B}, \textbf{D}) transmission delay. The trajectories converging to the same first binary pattern are plotted in the same color. The light-gray arrow in Figure~\ref{fig:Figure7}\textbf{A} indicates the point at which $t=\tau$. At this point the delay term $C_1\beta_K\mathbf{J}\tanh(\lambda\mathbf{u}(t-\tau))$ in the system (\ref{eq:PotentialDelay00}) ceases being fixed, and the trajectory starts to evolve towards the point corresponding to the binary pattern $(-,+,+)^T$ and labeled with a dark-gray arrow. Here, we avoid calling the point equilibrium, because although after the saddle-node bifurcation in the derived system (\ref{eq:IVPODE00}) an attracting node does exist around the point, it does not exist before the bifurcation. However, in the region around this point, the system behaves like that in the neighborhood of a saddle. The trajectory approaches the point, and then moves away from it. In contrast, in the network without delay, there is no such a period as in the network with delay before $t=\tau$ (the arch between the starting point and the point labelled with the light-gray arrow). The phase trajectories of the networks without delay directly approach the point at which a saddle-node on limit cycle bifurcation occurs when the parameter values move across the bifurcation curve. Therefore, in Figure~\ref{fig:Figure7}\textbf{B} along the purple representative phase trajectory, there is no point corresponding to that labelled with a light-gray arrow in Figure~\ref{fig:Figure7}\textbf{A}. This is the most significant difference between the corresponding phase trajectories in the two networks. Using the Matlab packages for numerical continuation computations and bifurcation analysis, DDE-BIFTOOL and Matcont, we verified that each of the two purple representative phase trajectories in Figure~\ref{fig:Figure7}\textbf{A} and \textbf{B} converges to an attracting limit cycle (see Section 3 for details). In Figure~\ref{fig:Figure7}\textbf{C} and \textbf{D} we illustrate the phase trajectories (180 trajectories in each panel) in the same two networks with the parameter $C_0$ increased from $0.75$ (\textbf{A} and \textbf{B}) to $0.76$. Clearly, both networks stop retrieving the prescribed cycle. In Section 3, we will see that a multiple saddle-nodes on limit cycle bifurcation \cite{Hoppensteadt97} occurs during the increase of $C_0$ in the network both with and without delay, which breaks the limit cycle corresponding to the prescribed cycle into 6 pairs saddles and nodes. In Figure~\ref{fig:Figure7}\textbf{E}, we illustrate the curves of the function (\ref{eq:RHSODE00}) with $\xi_i^{(\mu+1)}$ $=$ $1$ (the upper solid and dashed curves) and $-1$ (the lower solid and dashed curves) respectively. When $C_0$ increases from $0.75$ to $0.76$, the graphs of the functions (\ref{eq:RHSODE00}) move from the solid curves to the dashed curves, while the points $f_i(p_i)$ and $f_i(q_i)$ move across the horizontal axis, indicating that a saddle-node bifurcation occurs in the derived system (\ref{eq:IVPODE00}). In order to visualize the portion of the curves close to the horizontal axis more clearly, in Figure~\ref{fig:Figure7}\textbf{F} we illustrate the portion of the solid and dashed curves in the region enclosed by a dashed gray box in Figure~\ref{fig:Figure7}\textbf{E}. 
	}
	\label{ex:Example2}
\end{example}

\section{Bifurcations in Networks Constructed from Admissible Cycles}
In Section 2, we have proved that with appropriately chosen parameter values, every admissible cycle is retrievable. In this section, we show that the local bifurcation structures at the trivial equilibrium of the networks can be determined by the structural features of the cycles prescribed in the networks.

\subsection{Linear Stability and Scenarios for Possible Local Bifurcations of the Trivial Equilibrium Solution}
Consider the nonlinear system of delay differential equations (\ref{eq:PotentialDelay00}). For the convenience of analysis, we rescale the time variable $t$ by applying the change of coordinates $t=\tilde{t}\tau$ and writing $\mathbf{w}(\tilde{t}) = \mathbf{u}(\tilde{t}\tau)$ to get
$$
	\mathbf{\dot{w}}(\tilde{t}) = \tau\Big{(}-\mathbf{w}(\tilde{t}) + C_0\beta_K\mathbf{J^0}\tanh(\lambda\mathbf{w}(\tilde{t})) + C_1\beta_K\mathbf{J}\tanh(\lambda\mathbf{w}(\tilde{t}-1))\Big{)}.
$$
To simplify the notation when it does not cause any confusion, we replace $\tilde{t}$ by $t$. Thus, we get the rescaled system as follows
\begin{equation}
	\label{eq:PotentialDelayRescaled00}
	\mathbf{\dot{w}}(t) = \tau\Big{(}-\mathbf{w}(t) + C_0\beta_K\mathbf{J^0}\tanh(\lambda\mathbf{w}(t)) + C_1\beta_K\mathbf{J}\tanh(\lambda\mathbf{w}(t-1))\Big{)}.
\end{equation}

It is easy to verify that $\mathbf{w^*}=\mathbf{0}$ is an equilibrium solution of the above system. By expanding the right-hand side of the system into Taylor series in the neighborhood of the equilibrium solution $\mathbf{w^*} = \mathbf{0}$, and neglecting the nonlinear terms, the linearization of the rescaled system (\ref{eq:PotentialDelayRescaled00}) is derived as follows
\begin{equation}
	\label{eq:PotentialDelayLinearized00}
	\dot{\delta\mathbf{w}}(t) = \mathbf{A}_1\delta\mathbf{w}(t) + \mathbf{A}_2\delta\mathbf{w}(t-1).
\end{equation}
where $\mathbf{A}_1 = \tau(C_0\beta\mathbf{J^0} - \mathbf{I})$ and $\mathbf{A}_2 = \tau C_1\beta\mathbf{J}$. Suppose (\ref{eq:PotentialDelayLinearized00}) has a solution of the form $\delta\mathbf{w}(t) = e^{\sigma t}\phi$ with $\sigma\in\mathbb{C}$ and $\phi\in\mathbb{R}^N$, and accordingly $\delta\mathbf{w}(t-1) = e^{\sigma (t-1)}\phi$. Substituting these two functions back into (\ref{eq:PotentialDelayLinearized00}) gives
\begin{equation}
	\label{eq:PreCharacteristicEquation00}
	\sigma e^{\sigma t}\phi = \tau(C_0\beta\mathbf{J^0} - \mathbf{I})e^{\sigma t}\phi + \tau C_1\beta\mathbf{J} e^{\sigma(t-1)}\phi.
\end{equation}
Let $\Delta(\sigma)$ denote the \emph{characteristic matrix} of the linearized system (\ref{eq:PotentialDelayLinearized00}), then
$$
	\Delta(\sigma) = (\sigma + \tau)\mathbf{I} - \tau C_0\beta\mathbf{J^0} - \tau e^{-\sigma} C_1\beta\mathbf{J}.
$$
Thus, (\ref{eq:PreCharacteristicEquation00}) can be rewritten as $\Delta(\sigma)\phi = \mathbf{0}$, and it has a nontrivial solution if and only if 
\begin{equation}
	\label{eq:CharacteristicEquation00}
	\det(\Delta(\sigma)) = 0,
\end{equation}
which is called the \emph{characteristic equation}. The solution set of the characteristic equation (\ref{eq:CharacteristicEquation00}) forms the \emph{spectrum} of the \emph{infinitesimal generator} $\mathcal{A}$ of the strongly continuous semigroup $\{\mathcal{T}(t)|t\geq 0\}$ of solution maps $\mathcal{T}(t):\mathcal{C}\rightarrow\mathcal{C}$, where for every $t\geq 0$, the corresponding solution map $\mathcal{T}(t)$ is defined by the relation $\mathbf{w}_t(\theta) = \mathcal{T}(t)\varphi(\theta)$ with the initial data $\varphi(\theta)\in\mathcal{C}$ \cite{Diekmann95,Hale93}.

As the connectivity matrices $\mathbf{J^0}$ and $\mathbf{J}$ are constructed from the prescribed cycle $\Sigma$ with the pseudoinverse learning rule (\ref{eq:J0}) and (\ref{eq:J1}), we have that $\mathbf{J^0}$ and $\mathbf{J}$ commute \cite{Zhang13}, and accordingly, $\mathbf{J^0}$ and $\mathbf{J}$ are simultaneously diagonalizable \cite{Horn85}. Suppose $\mathbf{Q}$ is a nonsingular matrix which simultaneously diagonalizes the connectivity matrices $\mathbf{J^0}$ and $\mathbf{J}$. Thus, if the characteristic matrix has the diagonalization $\Delta(\sigma) = \mathbf{Q}\mathbf{K}\mathbf{Q}^{-1}$, where $\mathbf{K} = \mathrm{diag}(\kappa_1,\kappa_2,\dots,\kappa_N)$, then
$$
	\det(\Delta(\sigma)) = \det(\mathbf{K}) = \prod\limits_{i=1}^N \kappa_i
$$
and
$$
	\mathbf{K} = (\sigma + \tau)\mathbf{I} - \tau C_0\beta\mathbf{Q}^{-1}\mathbf{J^0}\mathbf{Q} - \tau e^{-\sigma} C_1\beta\mathbf{Q}^{-1}\mathbf{J}\mathbf{Q}.
$$

Suppose the prescribed admissible cycle $\Sigma$ is separable, minimal and consecutive, then $\mathbf{J^0} = \mathbf{I}$ \cite{Zhang13}, thus
$$
	\mathbf{K} = (\sigma + \tau(1 - C_0\beta))\mathbf{I} - \tau e^{-\sigma} C_1\beta\mathbf{Q}^{-1}\mathbf{J}\mathbf{Q}.
$$
Since $\mathbf{J}^p = \mathbf{I}$, it follows that $\tilde{\mathbf{K}} = \mathrm{diag}(\tilde{\kappa}_1, \tilde{\kappa}_2, \dots, \tilde{\kappa}_N)$ with $\tilde{\kappa}_i = e^{2n_i\pi\mathbf{i}/p}$, where $\tilde{\mathbf{K}} = \mathbf{Q}^{-1}\mathbf{J}\mathbf{Q}$, $\mathbf{i} = \sqrt{-1}$, $0\leq n_1 \leq n_2 \leq \cdots \leq n_N \leq p-1$, and $N\leq p$. Accordingly, we get
\begin{equation}
	\label{eq:Eigenvalues00}
	\kappa_i = \sigma + \tau(1 - C_0\beta) - \tau C_1\beta e^{-\sigma + 2n_i\pi\mathbf{i}/p}.
\end{equation}
Thus, the characteristic equation (\ref{eq:CharacteristicEquation00}) becomes
\begin{equation}
	\label{eq:CharacteristicEquationProd00}
	\prod\limits_{i=1}^N (\sigma - \tau C_1\beta e^{2n_i\pi\mathbf{i}/p - \sigma} + \tau(1 - C_0\beta)) = 0.
\end{equation}
Specially, for networks without transmission delay, the above characteristic equation reduces to
\begin{equation}
	\label{eq:CharacteristicEquationProdNoDelay00}
	\prod\limits_{i=1}^N (\sigma - C_1\beta e^{2n_i\pi\mathbf{i}/p} + (1 - C_0\beta)) = 0.
\end{equation}

Suppose $\sigma = \alpha + \mathbf{i}\omega$, $\alpha$, $\omega\in\mathbb{R}$, then following from (\ref{eq:CharacteristicEquationProd00}), we have that
\begin{equation}
	\label{eq:CharRootRealImagParts01}
	\begin{cases}
		\begin{array}{ccl}
			\alpha + \tau(1 - C_0\beta) & = & \tau(1 - C_0)\beta e^{-\alpha}\cos(\tilde{\omega}) \\
			         \omega             & = & \tau(1 - C_0)\beta e^{-\alpha}\sin(\tilde{\omega})
		\end{array}
	\end{cases}
\end{equation}
where $\tilde{\omega} = 2n_i\pi/p - \omega$. It follows that for $\alpha = 0$, 
\begin{equation}
	\label{eq:CharRootRealImagParts00}
	\begin{cases}
		\begin{array}{ccl}
			\tau(1 - C_0\beta) & = & \tau(1 - C_0)\beta\cos(\tilde{\omega}) \\
			\omega             & = & \tau(1 - C_0)\beta\sin(\tilde{\omega})
		\end{array}
	\end{cases}
\end{equation}
Thus, adding the first equation squared to the second equation squared gives
\begin{equation}
	\label{eq:CharRootC0}
	C_0=\frac{\beta+1}{2\beta} - \frac{\omega^2}{2\tau^2(\beta - 1)\beta}.
\end{equation}
Solving the first equation for $\omega$ gives
\begin{equation}
	\label{eq:CharRootOmega}
	\omega = \frac{2n_i\pi}{p} - \mathrm{arccos}\frac{1 - C_0\beta}{(1 - C_0)\beta}.
\end{equation}
Substituting (\ref{eq:CharRootOmega}) into (\ref{eq:CharRootC0}) gives the implicit equation for the curves of the characteristic roots with zero real part,
\begin{equation}
	\label{eq:StabilityBoundariesDelay00}
	C_0 - \frac{\beta + 1}{2\beta} + \frac{1}{2\tau^2(\beta - 1)\beta}\left(\frac{2n_i\pi}{p} - \mathrm{arccos}\frac{1 - C_0\beta}{(1 - C_0)\beta}\right)^2 = 0.
\end{equation}
In the case of networks without transmission delay, the above implicit equation reduces to
\begin{equation}
	\label{eq:StabilityBoundariesNoDelay00}
	C_0 = \frac{1 - \beta\cos(2n_i\pi/p)}{(1 - \cos(2n_i\pi/p))\beta}.
\end{equation}

\begin{figure}
	\begin{center}
		\includegraphics[height=2.8in]{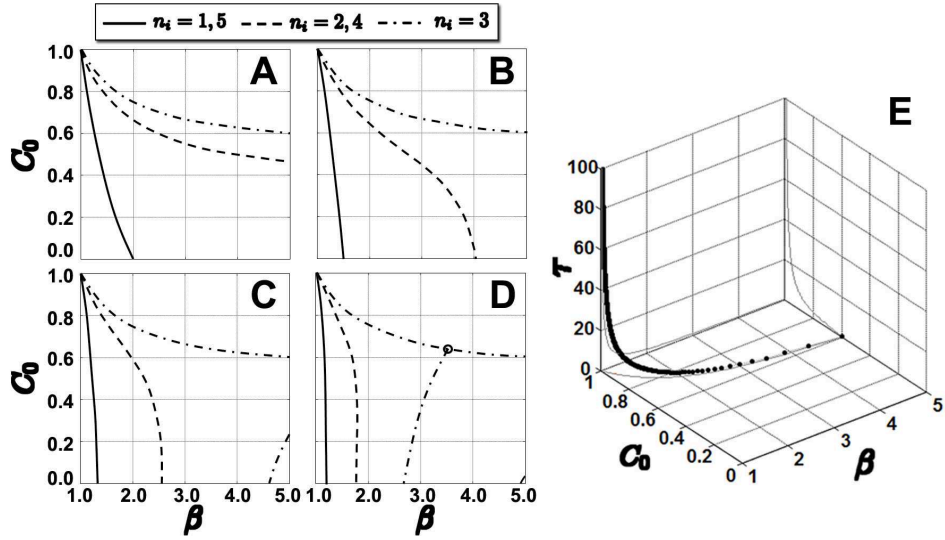}
	\end{center}
	\caption{\footnotesize{Curves of characteristic roots with zero real part (\textbf{A}-\textbf{D}) and the curve of the double-zero characteristic roots (\textbf{E}). In \textbf{A}, \textbf{B}, \textbf{C} and \textbf{D}, the solid curves correspond to $n_i=1$ or $5$; the dashed curves correspond to $n_i=2$ or $4$; and the dash-dot curves correspond to $n_i=3$. For $n_i=0$, $\Re(\sigma)>0$ for all $0\leq C_0\leq 1$ and $1<\beta$, therefore no curve corresponding to $n_i=0$ is shown for $\beta>1$ and $0\leq C_0\leq 1$. In both \textbf{C} and \textbf{D} the dash-dot curve consists of two parts. On the top part, the characteristic roots $\sigma$ move across the imaginary axis along the real axis, and on the right part, the characteristic roots $\sigma$ move across the imaginary axis off the real axis. The open circle in \textbf{D} is the intersection of the two parts of the dash-dot curve, which is a double zero and suggests a Bogdanov-Takens bifurcation \cite{Gerhard87,Campbell08,Fan12}. The dotted curve in \textbf{E} corresponds to the double zeros of the characteristic equation (\ref{eq:CharacteristicEquationProd00}). The parameters for the curves shown in this figure are set as follows, $p=6$, $\tau = 0$(\textbf{A}), $\tau = 0.2$(\textbf{B}), $\tau = 0.4$ (\textbf{C}), and $\tau = 0.8$(\textbf{D}).}}
	\label{fig:StabilityBoundary}
\end{figure}

Since $\mathbf{P}$ is a cyclic permutation matrix (\ref{eq:P}), eigenvalues of $\mathbf{P}$ are  $p$-th roots $\rho^{k} = e^{2k\pi\mathbf{i}/p}$ of unity, and its $(k+1)$-th eigenvector has the general form $v^{(k)} = (1,\rho^k,\dots,\rho^{(p-1)k})^T$. In \cite{Zhang13}, we have seen that the indices $n_i$ in the characteristic equations (\ref{eq:CharacteristicEquationProd00}) and (\ref{eq:CharacteristicEquationProdNoDelay00}) are determined by the structural features of $\Sigma$. Specifically, if the cycle $\Sigma$ annihilates the $(k+1)$-th eigenvector, i.e. $\Sigma v^{(k)} = \mathbf{0}$, then $k$ will not appear in the set of the $N$ indices $\{n_i| i = 1, 2, \dots, N\}$. Thus, the structural features of a cycle $\Sigma$ determine the linear stability of the equilibrium points of the networks constructed from it by selecting the indices $n_i$ that appear in the corresponding characteristic equation (\ref{eq:CharacteristicEquationProd00}) or (\ref{eq:CharacteristicEquationProdNoDelay00}), and hence indirectly determine the local dynamics of the networks with and without transmission delay.

Therefore, in general, for a given $p$, depending on how characteristic roots move across the imaginary axis on the curves defined by (\ref{eq:StabilityBoundariesDelay00}) or (\ref{eq:StabilityBoundariesNoDelay00}), a scenario of all possible local bifurcations of the trivial equilibrium solution of a network constructed from a simple cycle of length $p$ can be obtained. Next, as an illustrative example, we show the scenario of all possible local bifurcations of the trivial equilibrium solution of the networks constructed from simple cycles of length $p=6$ in Figure~\ref{fig:StabilityBoundary}. 

The curves in panels \textbf{A}, \textbf{B}, \textbf{C} and \textbf{D} respectively illustrate the characteristic roots with zero real part for the transmission delay $\tau = 0$, $0.2$, $0.4$, and $0.8$. The solid and dashed curves correspond to purely imaginary characteristic roots, and $n_i=1$ or $5$ and $n_i = 2$ or $4$ respectively. If the index $n_i=1,5$ or $n_i=2,4$ is ``chosen'' by the prescribed cycle, then when parameters of the network move across the corresponding curve, solid or dashed, transversely, a Hopf bifurcation would occur. 

The dash-dot curve on the top of each of the four panels corresponds to zero characteristic roots, and $n_i=3$. Since no quadratic term appears in the Taylor expansion of the right-hand side of (\ref{eq:PotentialDelayRescaled00}) around the trivial solution $\mathbf{w^*} = \mathbf{0}$, it follows that if the prescribed cycle ``selected'' $n_i=3$, and the parameters of the network move across the dash-dot curve on the top transversely, a pitchfork bifurcation would occur.

It is necessary to mention that the curve of zero characteristic roots is independent of the transmission delay $\tau$, and an easy direct calculation shows that the explicit formula for this curve is
\begin{equation}
	\label{eq:PitchforkCurve00}
	C_0 = \frac{1 + \beta}{2\beta},
\end{equation}
and in this case $n_i = p/2$, where $p$ is not only just $6$, it could be any even natural number with $p/2$ being odd. It is also not difficult to see that for $n_i=p/2$, in addition to the dash-dot curve (\ref{eq:PitchforkCurve00}) on the top in every panel \textbf{A}, \textbf{B}, \textbf{C} and \textbf{D}, the equation (\ref{eq:StabilityBoundariesDelay00}) has other solutions. These ``extra'' solutions correspond to purely imaginary characteristic roots too. In \textbf{C} and \textbf{D}, these ``extra'' solutions are plotted in dash-dot curves on the right. Accordingly, if $n_i=3$ is chosen, then when parameters $(C_0,\beta)$ move across the curve of these ``extra'' solutions transversely, a Hopf bifurcation would occur. As the ``extra'' solutions change with the transmission delay $\tau$, the curve of these solutions moves together with other curves corresponding to purely imaginary characteristic roots, and intersects with the dash-dot curve on the top at one single point (open circle in Figure~\ref{fig:StabilityBoundary}\textbf{D}). At this point, the curve of the first (i.e. leftmost) ``extra'' solution terminates. Since the intersection point corresponds to a double zero of the characteristic equation (\ref{eq:CharacteristicEquationProd00}), it can be verified through direct computations that at this point a codimension-two Bogdanov-Takens bifurcation occurs \cite{Gerhard87,Campbell08,Fan12}. The points on the dotted curve in \textbf{E} correspond to these double zeros.

In general, for networks constructed from admissible cycles with $p$ even, Figure~\ref{fig:StabilityBoundary} provides an overview of the scenario of all possible local bifurcations of the trivial equilibrium solution. If $n_i=p/2$ is ``chosen'', then at the trivial equilibrium solution of the corresponding network, Hopf bifurcations, pitchfork bifurcation, and Bogdanov-Takens bifurcation would happen. If $n_i=p/2$ is not ``chosen'', then at the trivial equilibrium solution, only Hopf bifurcations would happen. If $n_i=0$ is ``selected'', then all the solutions bifurcating from the trivial equilibrium, including the trivial equilibrium itself, are unstable.

Directly substituting $\omega=0$ into (\ref{eq:CharRootC0}) and (\ref{eq:CharRootOmega}) gives $n_i=p/2$. Since for networks constructed from admissible cycles with $p$ odd, $p/2$ is not an integer, it follows that no index $n_i$ can be $p/2$, and accordingly, when the characteristic roots move across the imaginary axis, none of them passes through the origin. Thus, in such networks, no pitchfork bifurcation would happen, only Hopf bifurcations occur at the trivial equilibrium solution.

Since every network constructed from a separable cycle $\Sigma$ consists of isolated clusters, and each of such clusters corresponds to a simple cycle associated with a generator of the prescribed cycle $\Sigma$ \cite{Zhang13}, it follows that the local bifurcations of the trivial solution of such networks are determined by the structural features of its simple cycle components.

For networks constructed from inseparable cycles, local bifurcation structures of the trivial equilibrium solution are much more complicated.

In the next subsection, we demonstrate how the structural features of the prescribed cycles determine the local bifurcation structures in examples. We demonstrate that prescribed cycles are stored and retrieved in the corresponding networks as different mathematical objects. Anti-symmetric simple MC-cycles \cite{Zhang13} of size $N\times p$ with $N = p/2$ are stored and retrieved as the attracting limit cycles created from the Hopf bifurcation which occurs when the pair of the conjuate complex characteristic roots with the largest real part move across the imaginary axis transversely from the left. Simple MC-cycles of size $N\times p$ with $N=p$ are stored and retrieved as transient oscillations that are purely due to the effects of the transmission delay. More complicated cycles, including simple cycles and inseparable composite cycles, both prescribed and derived, are stored and retrieved as either attracting limit cycles or transient oscillations induced by the delay.

\subsection{Bifurcations in Networks Constructed from Admissible Cycles}
In Section 3.1, we have explained how the structural features of the admissible cycle prescribed in a network can be used to determine the local bifurcation structures of the network. In this subsection, we discuss the structure of the local bifurcations at the trivial equilibrium of the networks constructed from different type of admissible cycles.

\paragraph{Bifurcations in Networks Constructed from Anti-symmetric Simple MC-Cycles with $N = p/2$}
Before discussing local bifurcations of the trivial solution of the networks constructed from anti-symmetric simple MC-cycles with $N = p/2$, we recall two definitions from \cite{Zhang13}.

\begin{definition}
	\label{def:SimpleCycle00}
	A cycle $\Sigma$ is called \emph{simple}, if it is generated by one single binary row vector $\eta$, in other words, its rows are cyclic permutations of a binary row vector $\eta$. That is, if $\eta_i$ is the $i$-th row of the cycle $\Sigma$, then $\eta_i = \eta\mathbf{P}^k$ for some $k\in\mathbb{N}$.
\end{definition}

\begin{definition}
	\label{def:AntiSymmetricSimpleMC-Cycle00}
	A binary row vector $\eta = (\eta_1,\eta_2,\dots,\eta_p) \in \{-1,1\}^p$ is said to be \emph{anti-symmetric}, if it has the following two properties: (a) $p$ is even; (b) $\eta = (\zeta,-\zeta)$, where $\zeta = (\zeta_1,\zeta_2,\dots,\zeta_{p/2})\in \{-1,1\}^{p/2}$. A cycle $\Sigma$ of size $N\times p$ is called \emph{anti-symmetric simple MC}, if it has the following four properties: (a) $\Sigma$ is simple; (b) $\eta_i$ is anti-symmetric for every $i$, where $\eta_i$ designates the $i$-th row of $\Sigma$; (c) $\mathrm{rank}(\Sigma) = N$; (d) $\eta_{i+1} = \eta_{i}\mathbf{P}$ for all $1\leq i<N$, where $\mathbf{P}$ is the cyclic permutation matrix defined by (\ref{eq:P}). 
\end{definition}

\begin{remark}
	\label{rem:Remark6}{\rm
		For example, the admissible cycle (\ref{eq:PrescribedCycle01}) discussed in Example~\ref{ex:Example2} is an anti-symmetric simple MC-cycle of size $N\times p$ with $N=p/2$. Gencic et al. \cite{Gencic90} have considered storage and retrieval of such cycles. Both numerical simulations and analog electronic circuit experiments demonstrated successful storage and retrieval of such cycles in Hopfield-type neural networks without delay. In \cite{Zhang13}, we showed that networks constructed from such cycles are rings of unidirectionally coupled neurons (see Figure~\ref{fig:NetworkTopology}). In such ring networks of $N$ neurons, the connection from the first neuron to the $N$-th neuron is inhibitory, and except for this connection, all other connections are excitatory. It has been well known that for ring networks of unidirectionally coupled neurons, if the number of inhibitory couplings is odd, the ring networks can generate sustained oscillations, and such ring networks have been widely used in different areas ranging from digital circuits for variable-frequency oscillations \cite{Gutierrez99} to models of nervous systems for generating rhythmic movements \cite{Collins94,Dror99}.
		
		In this section, we show that in such ring networks, the prescribed anti-symmetric simple MC-cycles with $N=p/2$ are stored and retrieved as attracting limit cycles. For the sake of better visualization, we choose the cycle (\ref{eq:PrescribedCycle01}) for illustrations.
	}
\end{remark}

\begin{figure}
	\begin{center}
		\includegraphics[height=1.5in]{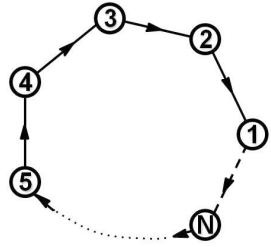}
	\end{center}
	\caption{\footnotesize{Topology of the networks constructed from anti-symmetric simple MC-cycles of size $N\times p$ with $p = 2N$. Such networks have topology of unidirectionally coupled neurons. All connections, except for the one from the neuron $1$ to the neuron $N$ which is inhibitory (dashed line), are excitatory (solid lines).}}
	\label{fig:NetworkTopology}
\end{figure}

\begin{figure}
	\begin{center}
		\includegraphics[height=4in]{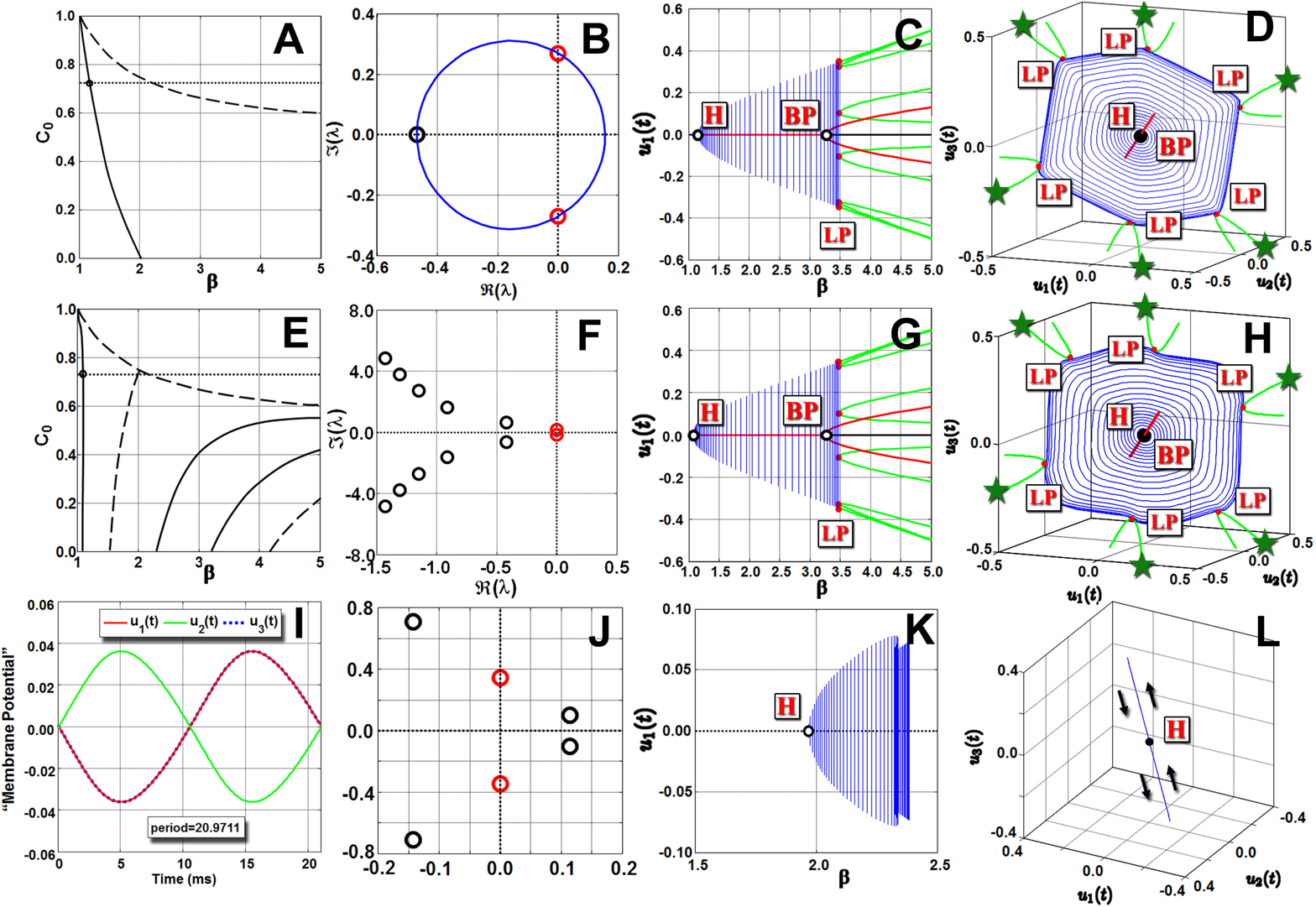}
	\end{center}
	\caption{\footnotesize{Local bifurcations of the trivial solution of the network constructed from the anti-symmetric simple MC-cycle (\ref{eq:PrescribedCycle01}). Panels from \textbf{A} to \textbf{D} are for the network without delay ($\tau=0$ms), and panels from \textbf{E} to \textbf{F} are for the network with delay ($\tau=2.0$ms). Panels \textbf{A} and \textbf{E} are bifurcation curves of the two networks respectively. In both panels, the solid curves are the bifurcation curves corresponding to $n_1=1$ and $n_3=5$, and the dashed curves are those of the bifurcations corresponding to $n_2=3$. The horizontal dotted line in each of \textbf{A} and \textbf{E} indicates the path ($C_0=0.73$) in parameter space along which the numerical continuation computations were carried out. The intersection between the dotted line and the solid curve indicates where the Hopf bifurcation occurs, and in both networks, this bifurcation creates the attracting limit cycle corresponding to the prescribed cycle (\ref{eq:PrescribedCycle01}). In both panels, the dashed curve on the top is the curve corresponding to the pitchfork bifurcation, and all other (dashed and solid) curves are Hopf bifurcation curves. Panels \textbf{B}, \textbf{F} and \textbf{J} illustrate distributions of characteristic roots when the conjugate pair of the characteristic roots with the largest (\textbf{B},\textbf{F}) and second (\textbf{J}) real part move across the imaginary axis from the left transversely, which indicates a Hopf bifurcation. Panels \textbf{C} and \textbf{D} illustrate the results of the numerical continuations of trivial and non-trivial equilibrium solutions and the periodic solution bifurcating from the trivial solution. In panels \textbf{C}, \textbf{D}, \textbf{G}, \textbf{H}, \textbf{K} and \textbf{L}, we adopt the notations of MatCont \cite{Kuznetsov03}, and use \textbf{H} to label \textbf{H}opf bifurcations, \textbf{BP} to label \textbf{B}ranch (pitchfork bifurcation) \textbf{P}oints, and \textbf{LP} to label \textbf{L}imit \textbf{P}oint (fold or saddle-node) bifurcations, respectively. The dark green pentagrams in panels \textbf{D} and \textbf{H} label the branches of nodes continued from the multiple saddle-nodes on limit cycle bifurcation. Panel \textbf{I} illustrates the unstable periodic solution ($\beta=2.0345$) bifurcating from the trivial solution via the first ``extra'' Hopf bifurcation corresponding to $n_2=3$, and \textbf{K} and \textbf{L} illustrate numerical continuations of this periodic solution. The continuation computations shown in \textbf{C} and \textbf{D} were implemented in MatCont 3.1. The continuation computations shown in \textbf{G} \textbf{H}, \textbf{K} and \textbf{L} were implemented in DDE-BIFTOOL 2.03.}}
	\label{fig:BifurRing}
\end{figure}

\begin{figure}
	\begin{center}
		\includegraphics[height=3.5in]{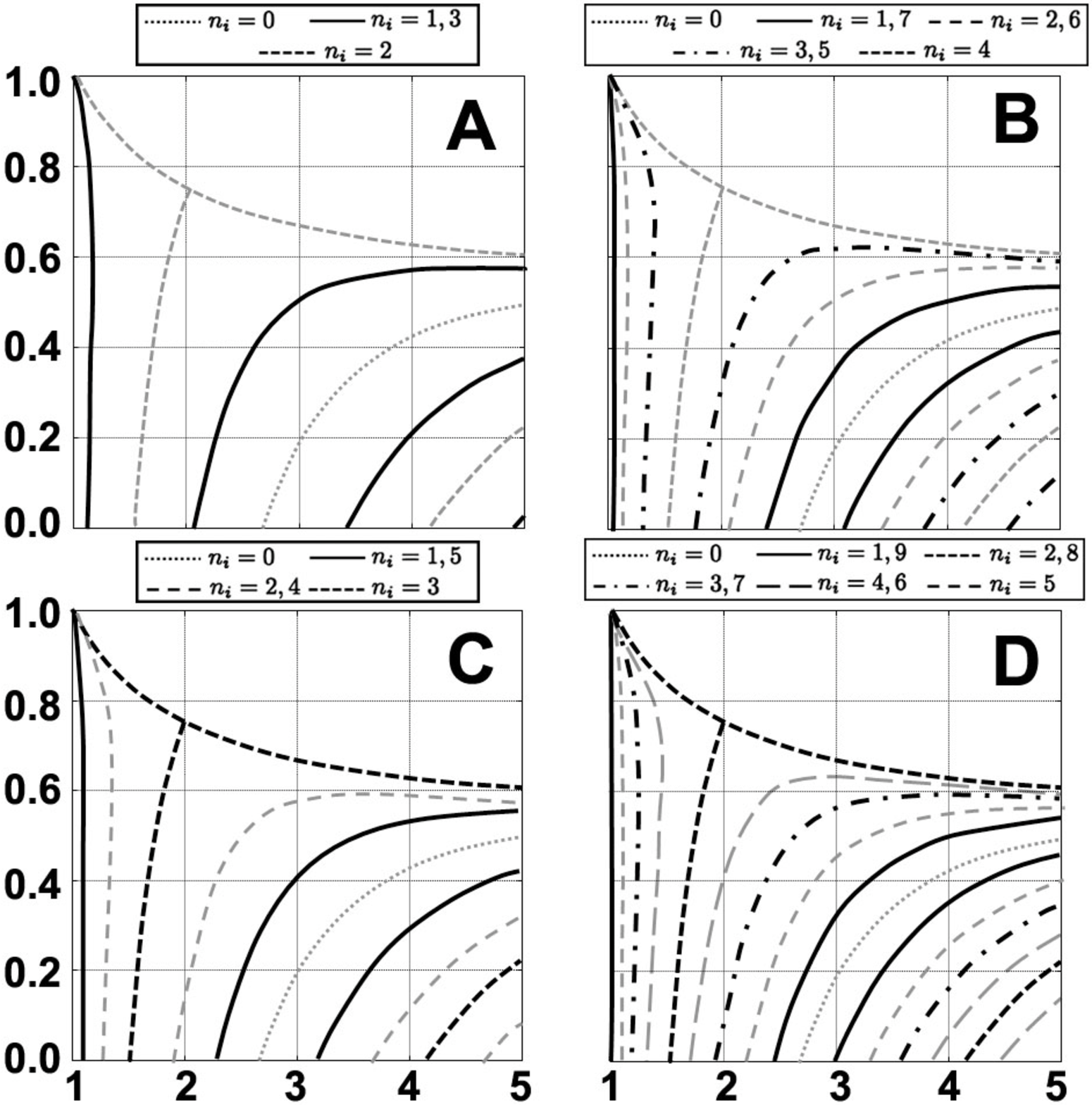}
	\end{center}
	\caption{\footnotesize{Scenario of all possible bifurcations of the networks constructed from a cycle of period $p=4$ (\textbf{A}), $8$ (\textbf{B}), $6$ (\textbf{C}), and $10$ (\textbf{D}) respectively. The curves in black and gray in each panel are the curves of characteristic roots of the network constructed from the four anti-symmetric simple MC-cycles with $N = p/2$, $\Sigma_1$ (\textbf{A}), $\Sigma_2$ (\textbf{B}), $\Sigma_3$ (\textbf{C}), and $\Sigma_4$ (\textbf{D}) (see text for the generators of these four cycles). The black curves in the four panels are those respectively chosen by the four prescribed cycles. For $\Sigma_1$ (\textbf{A}) and $\Sigma_2$ (\textbf{B}), the numbers of neurons in the two networks are $N = 2$ and $N = 4$. Clearly, in this case, $p/2$, which are $2$ and $4$ respectively, are not chosen by the prescribed cycles. Therefore, the pitchfork bifurcation (short dashed ``horizontal'' curve on the top of each panel) could not occur in these two networks, and accordingly, the Bogdanov-Takens bifurcation (intersection between the short dashed ``horizontal'' curve on the top and the short dashed ``vertical'' curve on the left of each panel) could not occur in these two networks either. For $\Sigma_3$ (\textbf{C}) and $\Sigma_4$ (\textbf{D}), the numbers of neurons in the two networks are $N = 3$ and $N = 5$. Thus, $p/2$, which are $3$ and $5$ respectively, are both chosen by the prescribed cycles. Therefore, in these two networks, both Hopf bifurcations and pitchfork bifurcation occur at the trivial equilibrium solution. Since the networks shown in this figure are with transmission delay ($\tau = 2.0$ms), it follows that the Bogdanov-Takens bifurcation occur at the trivial equilibrium solution in these two networks too. All the curves are obtained by numerically continuing Hopf bifurcations and steady-state bifurcations with DDE-BIFTOOL 2.03. Parameters for the computations are set as: $\lambda = 10$, $\tau = 2.0$.}}
	\label{fig:RingBifSummary00}
\end{figure}

\begin{example}
	\label{ex:Example3}{\rm
		Due to its anti-symmetric structures, the prescribe cycle $\Sigma$ annihilates the first, third, and fifth eigenvectors $v^{(0)}$ $=$ $(1$, $1$, $1$, $1$, $1$, $1)^T$, $v^{(2)}$ $=$ $(1$, $\rho^{2}$, $\rho^{4}$, $1$, $\rho^{2}$, $\rho^{4})^T$, and $v^{(4)}$ $=$ $(1$, $\rho^{4}$, $\rho^{2}$, $1$, $\rho^{4}$, $\rho^{2})^T$ of the cyclic permutation matrix $\mathbf{P}$, where $\rho = e^{\pi\mathbf{i}/3}$. Accordingly, it ``selects'' the indices $n_1 = 1$, $n_2 = 3$, and $n_3 = 5$ for the characteristic equation (\ref{eq:CharacteristicEquationProd00}) or (\ref{eq:CharacteristicEquationProdNoDelay00}) of the network constructed from it with and without delay respectively. Figure~\ref{fig:BifurRing}\textbf{A} shows the bifurcation curves of the network without delay. The solid curve is the Hopf bifurcation corresponding to $n_1=1$ and $n_3=5$, and the dashed curve is the pitchfork bifurcation corresponding to $n_2=3$. Since $0$, $2$, and $4$ are not the $n_i$ indices in the characteristic equation (\ref{eq:CharacteristicEquationProdNoDelay00}), it follows that the limit cycle created from the Hopf bifurcation corresponding to $n_1=1$ and $n_3=5$ may be stable, and the Hopf bifurcation corresponding to $n_i=2$ and $4$ could not occur in this network. Using the MatLab package MatCont 3.1, we numerically continue both the equilibrium solutions and periodic solution created from the trivial solution via the Hopf bifurcation. Both the analytic computations and numerical simulations confirm that the limit cycle created from the trivial solution via the Hopf bifurcation is stable. Numerical simulations (see Figure 1 in \cite{Zhang13}) show that the limit cycle satisfies the transition conditions imposed by the prescribed cycle (\ref{eq:PrescribedCycle01}). Moreover, numerical continuations (Figure~\ref{fig:BifurRing}\textbf{C} and \textbf{D}) also indicate that the prescribed cycle is retrieved as the attracting limit cycle (see the six nodes labelled by dark green pentagrams in panel \textbf{D}).
		
		Figure~\ref{fig:BifurRing}\textbf{E}-\textbf{H} illustrate the results of the same analysis implemented in the network with delay. From the distribution of characteristic roots (\textbf{F}) and bifurcation curves (\textbf{E}), it is not difficult to see that delay changes structure of the local bifurcations of the trivial solution significantly. Especially, due to the interaction between the Hopf bifurcation (\textbf{K},\textbf{L}) corresponding to $n_2=3$ and the complex conjugate characteristic roots with the second largest real part (\textbf{J}) and the pitchfork bifurcation, a codimension two Bogdanov-Takens bifurcation occurs in the network with delay, and this bifurcation can never happen in the network constructed from the cycle (\ref{eq:PrescribedCycle01}) without delay. In this network two cycles satisfy the transition conditions imposed by the prescribed cycle (\ref{eq:PrescribedCycle01}), one is the prescribed cycle, the other is the following derived cycle of periodic 2,
		$$
			\Sigma = \left(\begin{array}{cc}
				+ & - \\
				- & + \\
				+ & - 
			\end{array}\right).
		$$
		One can show that this derived cycle is retrieved as a unstable periodic solution, which does not exist in the network without delay. Figure~\ref{fig:BifurRing}\textbf{I} illustrates the profile of this unstable periodic solution at $C_0=0.73$ and $\beta=2.0345$. Figure~\ref{fig:BifurRing}\textbf{K} and \textbf{L} illustrate the numerical continuations of this periodic solution.
		
		Despite of the dramatic differences between the networks with and without delay, in terms of the ``principal'' Hopf bifurcation described above, which corresponds to $n_1=1$ and $5$, and the conjugate complex characteristic roots pair with the largest real part, delay does not change the qualitative structures of the local bifurcations of the trivial solution. Numerical simulations and continuation computation results (see panels \textbf{G} and \textbf{H}) show that a limit cycle is created from the trivial solution via the ``principal'' Hopf bifurcation. By checking the directin of the Hopf bifurcation, it can be shown that this limit cycle is stable, and by computing the overlap \cite{Gencic90,Zhang13}, it can be shown that this limit cycle corresponds to the prescribed cycle (\ref{eq:PrescribedCycle01}).
	}
\end{example}

\begin{remark}
	\label{rem:Remark7}{\rm
		In addition to the network constructed from the antisymmetric simple MC-cycle (\ref{eq:PrescribedCycle01}) both with and without delay, we also analyzed the networks constructed from other anti-symmetric simple MC-cycles $\Sigma_1$, $\Sigma_2$, and $\Sigma_4$ both with and without delay. The three cycles are respectively generated from the anti-symmetric binary row vectors $\eta_1$ $=$ $(1$, $1$, $-1$, $-1)$, $\eta_2$ $=$ $(1$, $1$, $1$, $1$, $-1$, $-1$, $-1$, $-1)$, and $\eta_4$ $=$ $(1$, $1$, $1$, $1$, $1$, $-1$, $-1$, $-1$, $-1$, $-1)$.
		
		In Figure~\ref{fig:RingBifSummary00} we summarize the structure of the local bifurcations of the trivial solution of the networks constructed from anti-symmetric simple MC-cycles with the four networks constructed from $\Sigma_1$, $\Sigma_2$, $\Sigma_3$ and $\Sigma_4$, where $\Sigma_3$ is the cycle (\ref{eq:PrescribedCycle01}) and has been analyzed in details in Example~\ref{ex:Example3}. 
		
		Here we consider the networks without delay as special case of the networks with delay. As for a fixed $p$, all local bifurcations of the trivial solution that may occur in a network constructed from an anti-symmetric simple MC-cycle with $N=p/2$, are those whose bifurcation curves correspond to the curves shown here to the left of the Hopf bifurcation curve that intersects with and terminates on the pitchfork bifurcation curve.
		
		All local bifurcation curves shown in Figure~\ref{fig:RingBifSummary00} are computed with $\tau = 2.0$ms. For bifurcations of individual networks, normal forms and unfoldings can be computed after reducing the networks onto their respective center manifolds. However, because these detailed technical computations are beyond our main goals of this paper, we would leave them elsewhere. Instead, here we summarize the structure of the local bifurcations at the trivial equilibrium solution of the networks constructed from anti-symmetric simple MC-cycles with $N = p/2$ as follows. Unless otherwise stated, the networks mentioned in the next two paragraphs are those constructed from anti-symmetric simple MC-cycles with $N = p/2$.
		
		If $N$ is even, only Hopf bifurcations occur at the trivial equilibrium solution. For networks without delay, only $N/2$ Hopf bifurcations occur at the trivial equilibrium solution, and the one corresponding to $n_i = 1$ and $p-1$ creates the attracting limit cycle corresponding to the prescribed cycle. For networks with delay, corresponding to each $n_i$, infinitely many Hopf bifurcations occur at the trivial equilibrium solution, and among those corresponding to $n_i = 1$ and $p-1$, the first creates the attracting limit cycle corresponding to the prescribed cycle.
		
		If $N$ is odd, both Hopf bifurcations and pitchfork bifurcation occur at the trivial equilibrium solution. For networks without delay, only $(N-1)/2$ Hopf bifurcations occur at the trivial equilibrium solution, and the one corresponding to $n_i = 1$ and $p-1$ creates the attracting limit cycle corresponding to the prescribed cycle. For networks with delay, infinitely many Hopf bifurcations occur at the trivial equilibrium solution, and among those corresponding to $n_i = 1$ and $p-1$, the first one creates the attracting limit cycle corresponding to the prescribed cycle. In networks both with and without delay, pitchfork bifurcation occurs when parameters move accross the curve (\ref{eq:PitchforkCurve00}), which is independent of the transmission delay $\tau$. For networks with delay, corresponding to $N = p/2$, in addition to the pitchfork bifurcation, infinitely many ``extra'' Hopf bifurcation occurs too. These two bifurcations interact at the point depicted in Figure~\ref{fig:StabilityBoundary}\textbf{E}, which leads to a Bogdanov-Takens bifurcation.
	}
\end{remark}

\paragraph{Bifurcations in Networks Constructed from Simple MC-Cycles with $N = p$}
In addition to the networks constructed from anti-symmetric simple MC-cycles with $N = p/2$, there is another type of networks which are rings of unidirectionally coupled neurons. They are networks constructed from simple MC-cycles with $N=p$. The only difference in terms of the network connections is that while every network constructed from anti-symmetric simple MC-cycles with $N=p/2$ has one inhibitory connection, all connections of the network constructed from anti-symmetric simple MC-cycles with $N=p$ are excitatory.

This type of excitatory unidirectional ring networks have been extensively investigated recently (see for example \cite{Pakdaman97,Guo07,Horikawa09a,Horikawa09b} etc.). Pakdaman et al. \cite{Pakdaman97} showed that the long lasting oscillations presented in such networks that they referred to as the \emph{transient oscillations} can not be explained by the analysis of the asymptotic behavior of the system. They considered the system of difference equations derived from the original system of delay differential equations. Such a system of difference equations can be used to approximate the original system of delay differential equations when the time scale under consideration is much larger than the characteristic charge-discharge time of the network. Pakdaman et al. showed that the long lasting oscillations presented in the original network correspond to the attracting periodic orbits in the descretized system of difference equations. Accordingly, they argued that the long lasting transient behavior observed in the original system of delay differential equations is due to the competition between the antagonistic asymptotic behavior of the original system and that of its descretized system. Both other properties of the transient oscillations and bifurcation structures of the excitatory unidirectional ring networks were investigated by many others in the past few years too (see for example \cite{Guo07,Horikawa09a,Horikawa09b} etc.).

Here, we may consider the networks investigated by Pakdaman et al \cite{Pakdaman97} and others as a special case of our networks constructed from simple MC-cycles with $N=p$, in which $C_0 = 0$. We claim that the cycles retrieved in such networks corresponding to the prescribed cycles are transient oscillations described by Pakdaman et al \cite{Pakdaman97}. Next, we briefly discuss in such networks how the prescribed cycles determine the local bifurcation structures of the trivial equilibrium solution.

Since the networks here are constructed from simple MC-cycles with $N=p$, all integers from $0$ to $p-1$ are indices appearing in the characteristic equations (\ref{eq:CharacteristicEquationProd00}) and (\ref{eq:CharacteristicEquationProdNoDelay00}) of the networks both with and without delay. As we have shown in Section 3.1, pitchfork bifurcation, Hopf bifurcations, and Bogdanov-Takens bifurcation may occur in such networks. From both the characteristic equations (\ref{eq:CharacteristicEquationProd00}) and (\ref{eq:CharacteristicEquationProdNoDelay00}), it is not difficult to see that for $n_i=0$, the characteristic root $\sigma=0$ only when $\beta = 1$, but by its definition $\beta = \mathrm{arctanh}(\beta_1)/\beta_1$ and $\beta_1\in (0,1)$, we have that $\beta > 1$ \cite{Zhang13}. It follows naturally that for $n_i=0$, throughout the whole parameter space, there is always one characteristic root with positive real part. Therefore, all the solutions bifurcating from the trivial equilibrium, including the trivial equilibrium itself, are unstable \cite{Hale93}. Accordingly, the periodic solutions bifurcating from the trivial equilibrium are unstable.

Next, we prove two useful results.

\begin{lemma}
	\label{lem:RingEquilibria00}
	Any network constructed from a simple MC-cycle with $N = p$ has at least three equilibrium solutions.
\end{lemma}

\begin{proof}
	Any network constructed from a simple MC-cycle has the following general form
	\begin{equation}
		\label{eq:ExcitatoryRingNetwork00}
		\frac{du_{i}}{dt}(t) = -u_i(t) + C_0\beta_K\tanh(\lambda u_i(t)) + C_1\beta_K\tanh(\lambda u_{i+1}(t - \tau))
	\end{equation}
	where $\tau=0$ for networks without delay, and following \cite{Pakdaman97}, the index $i$ is taken modulo $N+1$, i.e. $u_{N+1} = u_1$. We show that this network has at least the following three equilibria, $\mathbf{u^*}$ $=$ $\pm(u^*$, $u^*$, $\dots$, $u^*)^T$ $\in$ $\mathbb{R}^N$ and $\mathbf{0}$ $\in$ $\mathbb{R}^N$. It is trivial to verify that $\mathbf{0}$ is a solution. Therefore, we only consider the non-trivial solution here. Substitute the $\mathbf{u}(t)$ $=$ $\mathbf{u^*}$ into the network (\ref{eq:ExcitatoryRingNetwork00}) gives a scalar equation 
	$$
		\dot{u}^* = -u^* + \beta_K\tanh(\lambda u^*).
	$$
	Let $f(x) = x - \beta_K\tanh(\lambda x)$, we have that $\lim\limits_{x\rightarrow -\infty} f(x) = -\infty$, $\lim\limits_{x\rightarrow\infty} f(x) = \infty$, and $f(0) = 0$. Also, since $\dot{f}(x) = 1 - \beta(1 - \tanh^2(\lambda x))$, it follows that setting $\dot{f}(x) = 0$ gives 
	$$
		\tanh^2(\lambda x) = \frac{\beta - 1}{\beta}.
	$$
	Since $\beta>1$, we have that $f(x)$ has two distinct critical points $x^-$ and $x^+$. Since $f(0) = 0$, $\dot{f}(0) < 0$, and $\lim\limits_{x\rightarrow\pm\infty}\dot{f}(x)$ $=$ $1$ $>$ $0$ it follows from the intermediate value theorem that $f(x^-) > 0$ and $f(x^+) < 0$, and this implies that $f(x)$ has three roots. This proves the assertion.
\end{proof}

For the convenience of discussion, we denote the two non-trivial equilibrium solutions by $\mathbf{u^-}$ and $\mathbf{u^+}$ respectively. Next, we prove that these two non-trivial equilibrium solutions are asymptotically stable. 

\begin{lemma}
	\label{eq:RingEquilibriaStability00}
	Let $u^*>0$ be the positive solution of the equation $x - \beta_K\tanh(\lambda x) = 0$. Then the non-trivial equilibrium solutions $\mathbf{u^-}$ $=$ $-(u^*$, $u^*$, $\dots$, $u^*)^T$ $\in$ $\mathbf{R}^N$ and $\mathbf{u^+}$ $=$ $(u^*$, $u^*$, $\dots$, $u^*)^T$ $\in$ $\mathbf{R}^N$ are asymptotically stable for $\beta < 1/(1 - \tanh^2(\lambda u^*))$.
\end{lemma}

\begin{proof}
	Here we only analyze the local stability of the equilibrium $\mathbf{u^+}$, because the local stability analysis for $\mathbf{u^-}$ is exactly the same. 
	
	Linearizing the original network (\ref{eq:ExcitatoryRingNetwork00}) around the equilibrium $\mathbf{u^+}$ gives
	$$
		\dot{\mathbf{u}}(t) = \mathbf{A_1}(\mathbf{u^+})\mathbf{u}(t) + \mathbf{A_2}(\mathbf{u^+})\mathbf{u}(t-\tau)
	$$
	where $\mathbf{A_1}(\mathbf{u^+}) = (C_0\beta(1 - \tanh^2(\lambda u^*)) - 1)\mathbf{I}$ and $\mathbf{A_2}(\mathbf{u^+}) = C_1\beta(1 - \tanh^2(\lambda u^*))$. Substituting the ansatz $\mathbf{u} = \phi e^{\sigma t}$ with $\phi\in\mathbb{R}^N$ into the above linearized equation leads to the characteristic equation as follows
	$$
		\det(\Delta(\sigma)) = \prod\limits_{i=1}^N(\sigma + 1 - C_0\beta(1 - \tanh^2(\lambda u^*)) - C_1\beta(1 - \tanh^2(\lambda u^*)) e^{2n_i\pi\mathbf{i}/p - \sigma\tau}) = 0.
	$$
	For $n_i = 0$, the above characteristic equation reduces to
	$$
		\sigma + 1 - C_0\beta(1 - \tanh^2(\lambda u^*)) - C_1\beta(1 - \tanh^2(\lambda u^*)) e^{ - \sigma\tau} = 0.
	$$
	Let $\sigma = \alpha + \mathbf{i}\omega$, then the above equation becomes
	$$
		\begin{cases}
			\begin{array}{ccl}
				\alpha & = & (C_0 + C_1 e^{ - \alpha\tau}\cos(\omega\tau))\beta(1 - \tanh^2(\lambda u^*)) - 1 \\
				\omega & = & C_1\beta(1 - \tanh^2(\lambda u^*))e^{ - \alpha\tau}\sin(\omega\tau)
			\end{array}
		\end{cases}
	$$
	If $\alpha = 0$, then $\beta = 1/((C_0 + C_1\cos(\omega\tau))(1 - \tanh^2(\lambda u^*)))$. It follows that for $\beta < 1/(1 - \tanh^2(\lambda u^*))$, $\Re(\sigma) = \alpha < 0$.
\end{proof}

Thus, for networks constructed from simple MC-cycles, the situation is very similar to that Pakdaman et al described in \cite{Pakdaman97}, and the unstable limit cycles bifurcating from the trivial equilbrium solution will stay in the boundary between the respective basins of attraction of the two non-trivial equilibrium points $\mathbf{u^-}$ and $\mathbf{u^+}$, which has been proved to be a codimension one locally Lipschitz manifold containing the unstable equilibrium point $\mathbf{u} = \mathbf{0}$ and its stable manifold \cite{Pakdaman97}.

\paragraph{Bifurcations in Networks Constructed from More Complicated Admissible Cycles}
In the above two different cases, we have discussed the local bifurcations of the trivial solution of the two types of ring networks of unidirectionally coupled neurons. We showed that the prescribed cycles in the two different types networks are retrieved as different objects. In the networks with one inhibitory connection, i.e., in those constructed from anti-symmetric simple MC-cycles with $N = p/2$, the prescribed cycles are retrieved as the attracting limit cycles bifurcating from the trivial equilibrium via Hopf bifurcations. In the networks without inhibitory connection, i.e., in those constructed from simple MC-cycles with $N = p$, the prescribed cycles are retrieved as the so called long lasting \emph{transient oscillations}. In this subsubsection, we continue to discuss the local bifurcations of the trivial solution of the networks constructed from more general and complicated cycles, and show that, in more general cases, both the prescribed cycles and the derived cycles stored in the same networks may be retrieved as either attracting limit cycles or the long lasting transient oscillations.

\begin{figure}
	\begin{center}
		\includegraphics[height=2.5in]{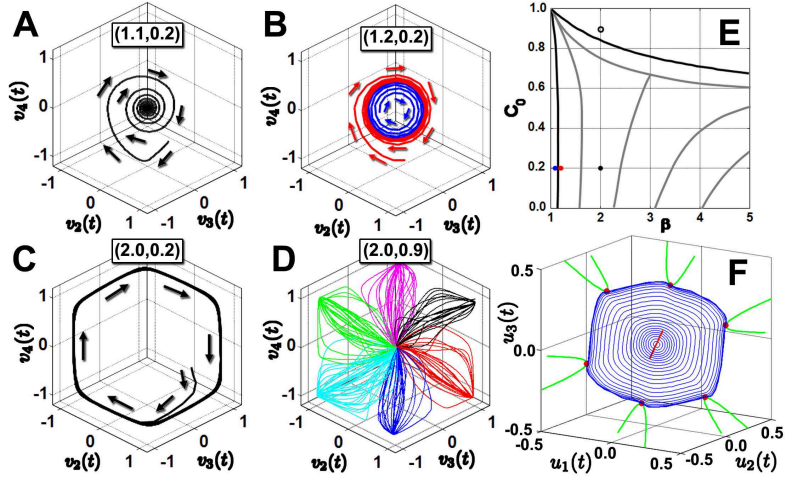}
	\end{center}
	\caption{\footnotesize{The attracting limit cycle bifurcating from the trivial equilibrium solution of the network constructed from $\Sigma_3$ in Example~\ref{ex:Example01}. In the four panels on the left, simulative solution trajectories starting from randomly chosen constant initial data with different $(\beta,C_0)$ parameter values (\textbf{A}: $(1.1,0.2)$, \textbf{B}: $(1.2,0.2)$, \textbf{C}: $(2.0,0.2)$, and \textbf{D}: $(2.0,0.9)$) are illustrated. The blue, red, black solid circles and the open circle in panel \textbf{E} indicate the locations of the $(\beta,C_0)$ parameter values for the simulations shown in the four panels on the left. The black curve on the left in \textbf{E} and the gray curves are curves on which conjugate complex and real characteristic roots move across the imaginary axis indicating possible Hopf and pitchfork bifurcations. The black curve on the top of \textbf{E} is the curve corresponding the \emph{multiple saddle-nodes on the limit cycle bifurcation} \cite{Hoppensteadt97}, which destroys the attracting limit cycle bifurcating from the trivial solution. Panel \textbf{F} shows the numerical continuations of the equilibrium solutions and periodic solution bifurcating from the trivial solution. The little red dots on the green curves are where the multiple saddle-nodes on limit cycle bifurcation occurs, and the green curves are branches of the saddles and nodes created from the destroyed limit cycle. The red line in the middle are the two branches of the symmetric equilibria $(u,-u,u,-u,u)$ and $(-u,u,-u,u,-u)$ respectively, and these two equilibria arise from the trivial equilibrium via the pitchfork bifurcation. Both the simulations (\textbf{A}-\textbf{D}) and the numerical continuations (\textbf{F}) confirm the occurence of the predicted bifurcations (\textbf{E}). The parameters are set as follows: $\tau = 1.0$ms, and $\lambda = 10$.}}
	\label{fig:N5P6Retrieval}
\end{figure}

Although in general, the network constructed from a generic simple cycle may have complicated network topology, the way the prescribed cycle determines the structures of local bifurcations of the trivial equilibrium solution of the network by its structural features remains exactly the same. Next, we illustrate in a network constructed from a simple cycle that the anti-symmetric derived cycle is retrieved as the attracting limit cycle, and the prescribed cycle is retrieved as the long lasting transient oscillations.

\begin{example}
	\label{ex:N5P6Bifurcations}{\rm
		Consider the cycle $\Sigma_3$ in Example~\ref{ex:Example01}. Figure~\ref{fig:N5P6Retrieval} illustrates the evolution of the attracting limit cycle bifurcating from the trivial equilibrium solution. Since $\Sigma_3$ only annihilates the first eigenvector $v^{(0)} = (1,1,\dots,1)^T\in\mathbb{C}^6$ of the cyclic permutation matrix $\mathbf{P}$, it follows that the indices $n_i$'s in the characteristic equation (\ref{eq:CharacteristicEquationProd00}) are $1,2,\dots,5$. Therefore, both Hopf bifurcations, and pitchfork bifurcation may occur. In Figure~\ref{fig:N5P6Retrieval}\textbf{E}, the curves on which these bifurcations may occur are shown as gray curves and the black curve on the left. The black curve on the left corresponds to the conjugate characteristic roots pair with the largest real part move across the imaginary axis transversely. Therefore, when parameters move across this curve, an attracting limit cycle may bifurcate from the trivial solution. In panels \textbf{A} and \textbf{B}, we numerically compute solution trajectories starting from randomly chosen constant initial data $\varphi(\theta)$ $\in$ $C([-1.0,0],\mathbb{R}^5)$, and in panel \textbf{F} we track both the periodic solution and the non-trivial equilibrium solutions corresponding to the binary patterns in the derived admissible cycle $\Sigma_1$ in Example~\ref{ex:Example01} with DDE-BIFTOOL. Both simulations (\textbf{A}-\textbf{D}) and numerical continuation computations (\textbf{F}) confirm the arising of the attracting limit cycle, and illustrate that the attracting limit cycle corresponds to the anti-symmetric, simple, and consecutive but not minimal cycle \cite{Zhang13} $\Sigma_1$ in Example~\ref{ex:Example01}. Meanwhile, (see Example~\ref{ex:Example3} too) as the Hopf bifurcation corresponding to $n_3=5=p/2$ does occur in this network, the cycle $\Sigma_4$ shown in Figure~\ref{fig:Figure5}\textbf{A} is retrieved successfully in this network too, and it is retrieved as the unstable periodic solution along the symmetric diagonal line $\{\mathbf{u}\in\mathbb{R}^5|\mathbf{u} = (u,-u,u,-u,u)\}$.
	}
\end{example}

In \cite{Zhang13}, we have seen that networks constructed from separable cycles consist of isolated clusters, and each cluster corresponds a simple cycle component. Accordingly, not only the local bifurcation structures of the trivial equilibrium solution, but also the dynamics of the networks are completely determined by its simple cycle components. Therefore, next we close our discussions on determination of the local bifurcation structures of the trivial equilibrium solution with a network constructed from an inseparable admissible cycle.

\begin{figure}
	\begin{center}
		\includegraphics[height=2in]{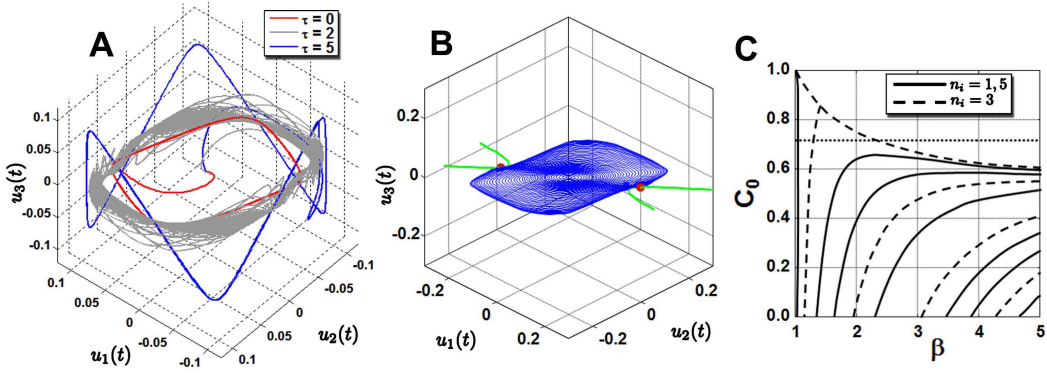}
	\end{center}
	\caption{\footnotesize{Retrieving inseparable composite cycle and local bifurcation structures. The curves in \textbf{A} are simulative solution trajectories with parameter values $C_0=0.71$, $\beta=2$ and different delay values: $\tau=0$ms (red curve), $\tau=2$ms (gray curve), $\tau=5$ms (blue curve). Panel \textbf{B} illustrates the numerical continuations (along $C_0=0.71$, see dotted line in \textbf{C}) of both the periodic solution (blue curves) bifurcating from the trivial equilibrium and the two pairs of saddles and nodes (green curves) bifurcating from the periodic solution via the multiple saddle-nodes on limit cycle bifurcation. The two saddle-nodes are plotted as the two small solid red circles on the two green curves. It is necessary to emphasize that all curves are actually on the plane $u_3(t) = 0$. Panel \textbf{C} illustrates curves of all possible local bifurcations of the trivial solution. All but the top dashed curve, which is that of the pitchfork bifurcation, are Hopf bifurcation curves. The numerical continuation computations and the bifurcation curves are implemented and obtained in Matlab with the package DDE-BIFTOOL 2.03.}}
	\label{fig:CompositeCycle00}
\end{figure}

\begin{example}
	\label{ex:InseparableBifurcation}{\rm
		Consider the network constructed from the following cycle
		\begin{equation}
			\label{eq:PrescribedCycle02}
			\Sigma = \left(\begin{array}{rrrrrr}
				 1 &  1 & -1 & -1 & -1 &  1  \\
				 1 & -1 & -1 & -1 &  1 &  1  \\
				-1 &  1 & -1 &  1 & -1 &  1 
			\end{array}\right).
		\end{equation}
		Since the vector space spanned by the set of all cyclic permutations of the third row $\eta_3$, which is called the loop generated by $\eta_3$ in \cite{Zhang13}, is contained in the vector space spanned by the set of all cyclic permutations of the first two rows $\eta_1$ and $\eta_2$, and $\eta_3$ is linearly independent of $\eta_1$ and $\eta_2$, it follows that $\Sigma$ is a inseparable composite MC-cycle \cite{Zhang13}.
		
		Direct computations show that the prescribed cycle is the only cycle satisfying the transition conditions imposed by the prescribed cycle itself. Figure~\ref{fig:CompositeCycle00} illustrates the simulative solutions (\textbf{A}), continuation of the limit cycle bifurcating from the trivial solution via the ``principal'' Hopf bifurcation (\textbf{B}), and the curves of all possible local bifurcations of the trivial solution (\textbf{C}). For $C_0=0.71$, $\beta=2.0$, when the delay time $\tau=0$ms, the solution trajectory approaches the attracting limit cycle created by the ``principal'' Hopf bifurcation. When increase $\tau$, the solution trajectory starts deviating from the attracting limit cycle (see the gray curve in panel \textbf{A} for example), and becomes more and more close to the prescribed cycle (see the blue curve in panel \textbf{A}). However, during this process, only one ``extra'' Hopf bifurcation occurs, which creats one unstable limit cycle, and four non-trivial equilibria are created via one pair of saddle-node bifurcations. We compared the successfully retrieved cycle, which corresponds to the prescribed cycle $\Sigma$, with both the attracting limit cycle created via the ``principal'' Hopf bifurcation corresponding to $n_1=1$ and $n_3=5$, and the unstable limit cycle created via the ``extra'' Hopf bifurcation corresponding to $n_2=3$. The retrieved cycle corresponds to none of them. We argue that this retrieved cycle may be the transient oscillation described by Pakdaman et al and others \cite{Pakdaman97,Horikawa09a,Horikawa09b}, and it may be the consequence of the competition among the attracting limit cycle, the unstable limit cycle and the saddles and nodes.
	}
\end{example}

\section{Conclusions and Discussions}
In summary, in this paper, we have systematically studied retrieval of admissible cycles in Hopfield-type networks with and without delay. In the networks with the $C_0$, $\beta$ parameter values appropriately chosen and the delay time $\tau$ sufficiently large in comparison to the time span of the onset/offset transients, we proved that any admissible cycle is retrievable. In terms of the linear stability analysis, we decompose each of the characteristic equations (\ref{eq:CharacteristicEquationProd00}) and (\ref{eq:CharacteristicEquationProdNoDelay00}) into a product of $N$ factors, each corresponds to an $n_i$ index in characteristic equation, which is an integer between $0$ and $p-1$. Based on this decomposition, we obtained a scenario of all possible local bifurcations of the trivial solution for every network constructed from an admissible cycle. Clearly, the scenario is determined by the prescribed admissible cycle in terms of its structural features by selecting the $n_i$ indices appearing in the characteristic equations (\ref{eq:CharacteristicEquationProd00}) and (\ref{eq:CharacteristicEquationProdNoDelay00}). In \cite{Zhang13}, we have explained how an admissible cycle determines which $N$ integers among those from $0$ to $p-1$ are chosen to be the $n_i$ indices. Since these $n_i$ indices determine the arrangement of the curves of characteristic roots with zero real part, which provides a scenario of all possible local bifurcations of the trivial solution, the prescribed cycle determines the structure of the local bifurcations of the trivial solution with its structural features by ``selecting'' the $n_i$ indices. In the context of networks of coupled oscillators, a similar idea has been used in determining stability of the synchronized oscillations \cite{Pecora98, Gabor12, Szalai13}. In \cite{Zhang13}, we have demonstrated a possible extension of our study to networks of coupled oscillators in a network of spiking neurons with bistable membrane behavior and postinhibitory rebound.

Using the MatLab packages, MatCont 3.1 and DDE-BIFTOOL 2.03, for numerical continuations and bifurcation analysis, we showed that addmissible cycles are stored and retrieved in the networks as different objects. Anti-symmetric simple MC-cycles with $N=p/2$ are stored and retrieved as the attracting limit cycles bifurcating from the trivial solution via the ``principal'' Hopf bifurcation. Anti-symmetric cycles of period two are stored and retrieved as the unstable periodic solution bifurcating from the trivial solution via the ``extra'' Hopf bifurcation corresponding to the index $n_i=p/2$, and the unstable periodic solution stays in the one-dimensional subspace corresponding to the symmetric diagonal $\{\mathbf{u}\in\mathbb{R}^N$ $|$ $\mathbf{u}$ $=$ $(u$, $-u$, $u$, $\dots$, $u)^T$, if $N$ is odd; and $\mathbf{u}$ $=$ $(u$, $-u$, $u$, $\dots$, $-u)^T$, if $N$ is even, with $u\in\mathbb{R}$ $\}$ in the phase space of the network. The rest admissible cycles are stored and retrieved as the long lasting transient oscillations. For $\tau$ sufficiently large, the transient oscillation last practically forever. 
	
While theoretical investigations \cite{Smith95, Guo07, Pakdaman97} suggested that solution trajectories of excitatory unidirectional ring networks should in general eventually converge to stable equilibria, numerical simulations usually show long lasting oscillatory patterns \cite{Pakdaman97}. Pakdaman et al and many others \cite{Pakdaman97, Horikawa09a, Horikawa09b} have studied such long lasting transient oscillations and their properties in details. It has been shown that such long lasting transient oscillations do not exist in the networks without delay, and can not be explained by the asymptotic dynamics of the networks with delay. In this paper, we illustrated that many admissible cycles are stored and retrieved in the networks with delay as transient oscillations, and based on the observations from the bifurcation analysis and numerical continuation computations of both the stable/unstable equilibrium solutions and periodic solutions, we conjecture that the transient oscillations are consequence of the interactions among the attracting limit cycle, unstable periodic solutions and equilibrium solutions. To clarify how the interactions among the stable/unstable periodic solutions and stable/unstable equilibrium solutions shape the long lasting transient oscillations would be a very interesting future direction to extend our study on storage and retrieval of cyclic patterns representing phase-locked oscillations discussed in this paper.
	
Cyclic patterns of neuronal activity in animal nervous systems are partially responsible for generating and controlling rhythmic movements from locomotion to gastrointestinal musculature activities. Neural networks of relatively small sizes that can produce cyclic patterned outputs without rhythmic sensory or central input are called central pattern generators (CPGs). So far, different models have been proposed to account for the underlying mechanisms of generation of the rhythmic activities \cite{Grillner85, Dickinson92, Buono01, Guertin09}. Among them, half-center oscillator is one of the most widely used models for studying CPGs \cite{McCrea08, Ermentrout09, Guertin09, Lewis13}, ring network model is another one \cite{Collins94, Dror99, Guertin09}. It has been shown that the classical half-center oscillator can be viewed as a limit cycle oscillator \cite{Cohen82}, so is the ring networks. During the past few decades, limit-cycle oscillators have played key roles in understanding the rhythmogenesis in animal CPG networks \cite{Buono01, Cohen92, Golubitsky99, Izhikevich07, Rubin09}. Recently, transient dynamics has been suggested to take important roles in generating cyclic patterns \cite{Rabinovich08}, and \emph{stable heteroclinic channels} \cite{Rabinovich01,Huerta04}, or \emph{stable heteroclinic sequences} \cite{Afraimovich04}. In this paper, we have illustrated that cyclic patterns except for those with special symmetric structures are stored and retrieved as transient oscillations, and those long lasting transient oscillations may be shaped by the interactions among the attracting limit cycle, unstable periodic solutions, saddles and nodes. Since by cyclically identifying stable submanifold of one saddle or hyperbolic periodic solution with the unstable submanifold of the other, a heteroclinic channel may be constructed, the interactions among the stable/unstable periodic solutions and stable/unstable equilibriam solutions may create heteroclinic channels, and some long lasting transient oscillations may correspond to such heteroclinic channels.



\begin{thebibliography}{1}

\bibitem{Agliari13} {\sc E. Agliari, A. Barra, A.~D. Antoni, and A. Galluzzi}, {\em Parallel retrieval of correlated patterns: from Hopfield networks to Boltzmann machines}, Neural Networks, 38 (2013), pp.~52--63.

\bibitem{Afraimovich04} {\sc V.~S. Afraimovich, V.~P. Zhigulin, and M.~I. Rabinovich}, {\em On the origin of reproducible sequential activity in neural circuits}, Chaos, 14 (2004), pp.~1123--1129.

\bibitem{Buono01} {\sc P.-L. Buono, and M. Golubitsky}, {\em Models of central pattern generators for quadruped locomotion}, J. Math. Biol. 42 (2001), pp.~291--326.


\bibitem{Campbell06} {\sc S.~A. Campbell, I. Ncube, and J. Wu}, {\em Multistability and stable asynchronous periodic oscillations in a multiple-delayed neural system}, Physica D, 214 (2006), pp.~101--119.

\bibitem{Campbell99} {\sc S.~A. Campbell, S. Ruan, and J. Wei}, {\em Qualitative analysis of a neural network model with multiple time delays}, Internat. J. Bifur. Chaos Appl. Sci. Engrg., 9 (1999), pp.~1585--1595.

\bibitem{Campbell08} {\sc S.~A. Campbell, and Y. Yuan}, {\em Zero singularities of codimension two and three in delay differential equations}, Nonlinearity, 21 (2008), pp.~2671--2691.

\bibitem{Campbell05} {\sc S.~A. Campbell, Y. Yuan, and S.~D. Bungay}, {\em Equivariant Hopf bifurcation in a ring of identical cells with delayed coupling}, Nonlinearity, 18 (2005), pp.~2827--2846.


\bibitem{Chen01} {\sc Y. Chen, and J. Wu}, {\em Existence and attraction of a phase-locked oscillation in a delayed network of two neurons}, Differ. Integral Equ., 14 (2001), pp.~1181--1236.

\bibitem{Shih06} {\sc C.~Y Cheng, K.~H Lin, and C.~W Shih}, {\em Multistability in recurrent neural networks}, SIAM J. Appl. Math., 66 (2006), pp.~1301--1320.

\bibitem{Shih07} {\sc C.~Y Cheng, K.~H Lin, and C.~W Shih}, {\em Multistability and convergence in delayed neural networks}, Physica D, 225 (2007), pp.~61--74.

\bibitem{Cohen82} {\sc A.~H. Cohen, P.~J. Holmes, and R.~H. Rand}, {\em The nature of the coupling between segmental oscillators of the lamprey spinal generator for locomotion}, J. Math. Biol., 13 (1982), pp.~345--369.

\bibitem{Cohen92} {\sc A.~H. Cohen, G.~B. Ermentrout, T. Kiemel, N. Kopell, K.~A. Sigvardt, and T.~L. Williams}, {\em Modelling of intersegmental coordination in the lamprey central pattern generator for locomotion}, Trends. Neurosci., 15 (1992), pp.~434-438.

\bibitem{Collins94} {\sc J.~J. Collins, and I. Stewart}, {\em A group-theoretic approach to rings of coupled biological oscillators}, Biol. Cybern., 71 (1994), pp.~95--103.


\bibitem{Gerhard87} {\sc G. Dangelmayr, and J. Guckenheimer}, {\em On a four parameter family of planar vector fields}, Arch. Rat. Mech. Anal., 97 (1987), pp.~321--352.

\bibitem{Kuznetsov03} {\sc A. Dhooge, W. Govaerts, and Y.~A. Kuznetsov}, {\em MatCont: A MATLAB package for numerical bifurcation analysis of ODEs}, ACM TOMS, 29 (2003), pp.~141--164.

\bibitem{Dickinson92} {\sc P.~S. Dickinson, and M. Moulins}, {\em Interactions and combinations between different networks in the stomatogastric nervous system}, in Dynamic Biological Networks: The Stomatogastric Nervous System, R.~M. Harris-Warrick, E. Marder, A.~I. Selverston, and M. Moulins, eds., Cambridge, MA: MIT, 1992, pp.~139--160.


\bibitem{Diekmann95} {\sc O. Diekmann, S.~A. Gils, S.~M.~V. Lunel, and H.~O. Walther}, {\em Delay Equations {\rm -} Functional{\rm -,} Complex{\rm -,} And Nonlinear Analysis}, Springer-Verlag, New York, 1995.

\bibitem{Driver77} {\sc R.~D. Driver}, {\em Ordinary And Delay Differential Equations}, Springer-Verlag, New York, pp.~226--240, 1977.

\bibitem{Dror99} {\sc R.~O. Dror, C.~C. Canavier, R.~J. Butera, J.~W. Clark, and J.~H. Byrne}, {\em A mathematical criterion based on phase response curves for stability in a ring of coupled oscillators}, Biol. Cybern., 80 (1999), pp.~11--23.

\bibitem{Fan12} {\sc G. Fan, S.~A. Campbell, G.~S.~K. Wolkowicz, and H. Zhu}, {\em The bifurcation study of 1:2 resonance in a delayed system of two coupled neurons}, J. Dynamics Differential Equations, 25 (2013), pp.~193--216.


\bibitem{Gencic90} {\sc T. Gencic, M. Lappe, G. Dangelmayr, and W. Guettinger}, {\em Storing cycles in analog neural networks}, in Parallel Processing In Neural Systems And Computers, R. Eckmiller, G. Hartmann and G. Hauske, eds., North Holland, 1990, pp.~445--450.

\bibitem{Glyzin13} {\sc S.~D. Glyzin, A.~Yu. Kolesov, and N.~Kh. Rozov}, {\em Relaxation self-oscillations in Hopfield networks with delay}, Izvestiya: Mathematics, 77 (2013), pp.~271--312.

\bibitem{Golubitsky99} {\sc M. Golubitsky, I. Stewart, P.~L. Buono, and J.~J. Collins}, {\em Symmetry in locomotor central pattern generators and animal gaits}, Nature, 401 (1999), pp.~693--695.

\bibitem{Grillner85} {\sc S. Grillner, and P. Wall\'en} {\em Central pattern generators for locomotion, with special reference to vertebrates}, Ann. Rev. Neurosci., 8 (1985), pp.~233--261.


\bibitem{Guertin09} {\sc P.~A. Guertin}, {\em The mammalian central pattern generator for locomotion}, Brain Res. Rev., 62 (2009), pp.~45--56.

\bibitem{Guo07} {\sc S.~J. Guo, and L.~H. Huang}, {\em Pattern formation and continuation in a trineuron ring with delays}, Acta Math. Sin. (Engl. Ser.), 23 (2007), pp.~799--818.

\bibitem{Gutierrez99} {\sc G. Gutierrez}, {\em Variable-frequency oscillators}, in Wiley Encyclopedia of Electrical and Electronics Engineering, J.~G. Webster, ed., John Wiley, New York, 1999, pp.~75--84.

\bibitem{Hale93} {\sc J.~K. Hale, and S.~M.~V. Lunel}, {\em Introduction To Functional Differential Equations}, Springer-Verlag, New York, 1993.



\bibitem{Hopfield84} {\sc J.~J. Hopfield}, {\em Neurons with graded response have collective computational properties like those of two-state neurons}, Proc. Natl. Acad. Sci. USA, 81 (1984), pp.~3088--3092.

\bibitem{Hoppensteadt97} {\sc F.~C. Hoppensteadt, and E.~M. Izhikevich}, {\em Weakly Connected Neural Networks}, Springer-Verlag, New York, 1997, pp.~72--76.

\bibitem{Horikawa09a} {\sc Y. Horikawa, and H. Kitajima}, {\em Duration of transient oscillations in ring networks of unidirectionally coupled neurons}, Physica D, 238 (2009), pp.~216--225.

\bibitem{Horikawa09b} {\sc Y. Horikawa, and H. Kitajima}, {\em Effects of noise and variations on the duration of transient oscillations in unidirectionally coupled bistable ring networks}, Phys. Rev. E, 80 (2009), pp.~021934-1--15.

\bibitem{Horn85} {\sc R.~A. Horn, and C.~R. Johnson}, {\em Matrix Analysis}, Cambridge University Press, New York, 1985.

\bibitem{Huerta04} {\sc R. Huerta, and M. Rabinovich}, {\em Reproducible sequence generation in random neural ensembles}, Phys. Rev. Lett., 93 (2004), pp.~238104-1--4.

\bibitem{Izhikevich07} {\sc E.~M. Izhikevich}, {\em Dynamical Systems in Neuroscience: The Geometry of Excitability and Bursting}, Computational Neuroscience, MIT Press, Cambridge, Massachusetts, 2007.

\bibitem{Jean01} {\sc A. Jean}, {\em Brain stem control of swallowing neuronal network and cellular mechanisms}, Physiol. Rev., 81 (2001), pp.~929--969.


\bibitem{Kleinfeld86} {\sc D. Kleinfeld}, {\em Sequential state generation by model neural networks}, Proc. Natl. Acad. Sci. USA, 83 (1986), pp.9469--9473.

\bibitem{Kleinfeld88} {\sc D. Kleinfeld, and H. Sompolinsky}, {\em Associative neural network model for the generation of temporal patterns, theory and application to central pattern generators}, Biophys. J., 54 (1988), pp.~1039--1051.

\bibitem{Ermentrout09} {\sc T.~K. Ko, and G.~B. Ermentrout}, {\em Phase-response curves of coupled oscillators}, Phys. Rev. E, 79 (2009), pp.~016211-1--016211-6.

\bibitem{Kohonen76} {\sc T. Kohonen, E. Reuhkala, K. M\"akisara, and L. Vainio}, {\em Associative recall of images}, Biol. Cybernet., 22 (1976), pp.~159--168.

\bibitem{MackayLyons02} {\sc M. MacKay-Lyons}, {\em Central pattern generation of locomotion: a review of the evidence}, Phys. Ther., 82 (2002), pp.~69--83.

\bibitem{Marder05} {\sc E. Marder, D. Bucher, D.~J. Schulz, and A.~L. Taylor}, {\em Invertebrate central pattern generation moves along}, Curr. Biol., 15 (2005), pp.~R685--R699.


\bibitem{McCrea08} {\sc D.~A. McCrea, and I.~A. Rybak} {\em Organization of mammalian locomotor rhythm and pattern generation}, Brain Res. Rev., 57 (2008), pp.~134--146.

\bibitem{Meyrand94} {\sc P. Meyrand, J. Simmers, and M. Moulins}, {\em Dynamic construction of a neural network from multiple pattern generators in the lobster stomatogastric nervous system}, J. Neurosci., 14 (1994), pp.~630--644.


\bibitem{Tani08} {\sc J. Namikawa, and J. Tani}, {\em A model for learning to segment temporal sequences, utilizing a mixture of RNN experts together with adaptive variance}, Neural Networks, 21 (2008), pp.~1466-1475

\bibitem{Lai04} {\sc T. Nishikawa, F.~C. Hoppensteadt, and Y.-C. Lai}, {\em Oscillatory associative memory network with perfect retrieval}, Physica D, 197 (2004), pp.~134--148.

\bibitem{Gabor12} {\sc G. Orosz}, {\em Decomposing the dynamics of delayed networks: equilibria and rhythmic patterns in neural systems}, Time Delay Systems, 10 (2012), pp.~173--178.

\bibitem{Pakdaman97} {\sc K. Pakdaman, C.~P. Malta, C. Grotta-Ragazzo, O. Arino, and J.~F. Vibert}, {\em Transient oscillations in continuous-time excitatory ring neural networks with delay}, Phys. Rev. E, 55 (1997), pp.~3234--3248.

\bibitem{Pasemann95} {\sc F. Pasemann} {\em Characterization of periodic attractors in neural ring networks}, Neural Networks, 8 (1995), pp.~421--429.

\bibitem{Pecora98} {\sc L.~M. Pecora, and T.~L. Carroll}, {\em Master stability functions for synchronized coupled systems}, Phys. Rev. Lett., 80 (1998), pp.~2109--2112.



\bibitem{Personnaz86} {\sc L. Personnaz, I. Guyon, and G. Dreyfus}, {\em Collective computational properties of neural networks: new learning mechanisms}, Phys. Rev. A, 34 (1986), pp.~4217--4228.

\bibitem{Petersen03}{\sc C.~C.~H. Petersen, T.~T.~G. Hahn, M. Mehta, A. Grinvald, and B. Sakmann}, {\em Interaction of sensory responses with spontaneous depolarization in layer 2/3 barrel cortex}, Proc. Natl. Acad. Sci. USA, 100 (2003), pp.~13638--13653.


\bibitem{Rabinovich08} {\sc M. Rabinovich, R. Huerta, and G. Laurent}, {\em Transient dynamics for neural processing}, Science, 321 (2008), pp.~48--50.

\bibitem{Rabinovich01} {\sc M. Rabinorich, A. Volkovskii, P. Lecanda, R. Huerta, H.~D.~I. Abarbanel, and G. Laurent}, {\em Dynamical encoding by networks of competing neuron groups: winnerless competition}, Phys. Rev. Lett. 87 (2001), pp.~068102-1--4

\bibitem{Rana02} {\sc I.~K. Rana}, {\em An Introduction to Measure and Integration}, 2nd ed., Graduate Studies in Mathematics, Vol. 45, American Mathematical Society, Providence, Rhode Island, 2002.


\bibitem{Rubin09} {\sc J.~E. Rubin, N.~A. Shevtsova, G.~B. Ermentrout, J.~C. Smith, and I.~A. Rybak}, {\em Multiple rhythmic states in a model of the respiratory central pattern generator}, J. Neurophysiol., 101 (2009), pp.~2146--2165.


\bibitem{Selverston10} {\sc A.~I. Selverston}, {\em Invertebrate central pattern generator circuits}, Philos. Trans. R. Soc. Lond. Ser. B Biol. Sci., 365 (2010), pp.~2329--2345.


\bibitem{Smith95} {\sc H.~L. Smith}, {\rm Monotone Dynamical Systems An Introduction to the Theory of Competitive and Cooperative Systems}, Mathematical Surveys and Monographs, vol. 41, American Mathematical Society, Providence, 1995.

\bibitem{Sompolinsky86} {\sc H. Sompolinsky, and I. Kanter}, {\em Temporal association in asymmetric neural networks}, Phys. Rev. Lett., 57 (1986), pp.~2861--2864.



\bibitem{Syed90} {\sc N.~I. Syed, A.~G.~M. Bulloch, and K. Lukowiak}, {\em \emph{In vitro} reconstruction of the respiratory central pattern generator of the mollusk \emph{Lymnaea}}, Science, 250 (1990), pp.~282--285.

\bibitem{Syed91} {\sc N.~I. Syed, D. Harrison, and W. Winlow}, {\em Respiratory behavior in the pond snail \emph{Lymnaea stagnalis}. I. behavioral analysis and the identification of motor neurons}, J. Comp. Physiol. A, 169 (1991), pp.~541--555.

\bibitem{Szalai13} {\sc R. Szalai, and G. Orosz}, {\em Decomposing the dynamics of heterogeneous delayed networks with applications to connected vehicle systems}, arxiv:1305.6771v3[nlin.AO], 2013.

\bibitem{Tang10} {\sc H. Tang, H. Li, and R. Yan}, {\em Memory dynamics in attractor networks with saliency weights}, Neural Comput., 22 (2010), pp.~1899--1926.



\bibitem{Yuan04} {\sc Y. Yuan and S.~A. Campbell} {\em Stability and synchronization of a ring of identical cells with delayed coupling}, J. Dyn. Diff. Equ, 16 (2004) pp.~709--744

\bibitem{Yuste08} {\sc R. Yuste}, {\em Circuit neuroscience: the road ahead}, Front. Neurosci., 2 (2008), pp.~6--9.

\bibitem{Zhang13} {\sc C. Zhang, G. Dangelmayr, and I. Oprea}, {\em Storing cycles in Hopfield-type networks with pseudoinverse learning rule - admissibility and network topology}, Neural Networks, 46 (2013), pp.~283--298.

\bibitem{Zhang14} {\sc C. Zhang, G. Dangelmayr, and I. Oprea}, {\em Delay induced long-lasting transient oscillations in Hopfield-type neural networks}, Preprint, (2014)

\bibitem{Lewis13} {\sc C. Zhang, and T.~J. Lewis}, {\em Phase response properties of half-center oscillators}, J. Comput. Neurosci., 35 (2013), pp.~55--74.


\end{thebibliography}
\end{document}